\documentclass[pdflatex,sn-mathphys-ay]{sn-jnl}

\usepackage{graphicx}%
\usepackage{multirow}%
\usepackage{amsmath,amssymb,amsfonts}%
\usepackage{amsthm}%
\usepackage[title]{appendix}%
\usepackage{xcolor}%
\usepackage{textcomp}%
\usepackage{manyfoot}%
\usepackage{booktabs}%
\usepackage{algorithm}%
\usepackage{algorithmicx}%
\usepackage{algpseudocode}%
\usepackage{listings}%

\newcommand{\Prob}{\mathsf{P}}
\newcommand{\Ex}{\mathsf{E}}
\newcommand{\Var}{\mathsf{D}}
\newcommand{\Real}{\mathbb{R}}
\newcommand{\ind}{\mathbb{I}}
\DeclareMathOperator{\cov}{cov}

\newenvironment{proofsketch}{\par\noindent\textit{Sketch of proof.}\ }{\qed\par}

\theoremstyle{thmstyleone}%
\newtheorem{theorem}{Theorem}
\newtheorem{proposition}[theorem]{Proposition}%

\theoremstyle{thmstyletwo}%
\newtheorem{remark}{Remark}%

\theoremstyle{thmstylethree}%

\newtheorem{lemma}[theorem]{Lemma}
\newtheorem{corollary}[theorem]{Corollary}%

\begin{document}

\title[Closed-form estimation under two-sided informative censorship]{Closed-form estimation of the exponential scale under two-sided informative random censorship}


\author*[1]{\fnm{Dilshod R.} \sur{Mansurov}}\email{dmansurovmath@gmail.com}

\author[2]{\fnm{Sukhrob B.} \sur{Bozorov}}\email{suxrobbek\_8912@mail.ru}

\author[1]{\fnm{Azizbek B.} \sur{Oltiyev}}\email{azizbekolot1993@gmail.com}

\affil*[1]{\orgdiv{Department of Mathematics}, \orgname{Navoi State University}, \orgaddress{\city{Navoi}, \postcode{210100}, \country{Uzbekistan}}}

\affil[2]{\orgdiv{Department of Mathematics}, \orgname{Gulistan State University}, \orgaddress{\city{Gulistan}, \postcode{120100}, \country{Uzbekistan}}}


\abstract{We estimate the scale of an exponential lifetime from observations censored
	randomly from both sides, under an informative model in which each censoring
	law is a power of the lifetime survival function -- a two-sided generalization
	of the Koziol--Green proportional hazards model. Substituting the empirical
	distribution function into the likelihood equation yields a closed-form
	pseudo-maximum-likelihood estimator of the scale that is computed in one pass
	over the order statistics
	and needs neither iteration nor numerical optimization. A digamma identity
	shows the estimating equation to be \emph{exactly}, not merely asymptotically,
	Fisher-consistent. Strong consistency and asymptotic normality are established
	with \emph{no restriction on the censoring depth}; the key is an identity that
	rewrites the underlying $L$-statistic without its unbounded score function and
	so removes the condition that classical limit theorems would impose. The influence
	function and the asymptotic variance are explicit in polygamma functions, so
	confidence intervals require no numerical integration. Against the information
	bound of the observed-data model the efficiency exceeds $98.6\%$ in the designs
	considered and stays above $96.2\%$ over a wide sweep. Only the entry half of
	the model is needed: the estimator remains exactly Fisher-consistent and
	strongly consistent whatever the right-censoring law, and exponentiality is
	shown to be the precise price of that freedom. The construction extends
	verbatim to a proportional hazards class with known baseline. Three data sets
	illustrate the procedure, including a cohort in which left censoring is
	genuine.}




\maketitle

\section{Introduction}

In biomedical follow-up and in engineering reliability testing the lifetime of
interest is often observed incompletely from \emph{both} sides. A unit enters
observation only at a random time $L$, and if the event of interest has already
occurred before that time we learn not its exact value but only that it fell
short of $L$; this is \emph{random censorship from the left}. At the same time
the follow-up may be terminated prematurely by a right-censoring variable $Y$
(loss to follow-up, withdrawal from the trial, the end of the study), which is
\emph{random censorship from the right}. Under censorship acting from both
sides the observable quantity is therefore
$Z=\max\{L,\min(X,Y)\}$, the lifetime $X$ being censored on two sides by the
random pair $(L,Y)$.

It matters to distinguish left censoring from \emph{left truncation} (delayed
entry). Under truncation a unit whose event precedes $L$ never enters the
sample at all and the observed law is conditioned on $\{V\ge L\}$. In the
present scheme such a unit remains in the sample and carries, through the
indicator $\delta^{(0)}=1$, the information that the event occurred before $L$;
consequently the distribution function (df) of $Z$ stays unconditional,
$H=KN$. This structure arises, for instance, when a unit is first inspected at
time $L$ and found to have already failed, or when a measurement falls below a
limit of detection. Left truncation, frequently together with interval
censoring, has a literature of its own \citep{Sh21}, and none of the results
below applies to it. Throughout, $L$, $X$ and $Y$ are assumed independent with
continuous dfs $K$, $F$ and $G$ respectively. Standard accounts of censored-data
methodology are given by \citet{KM03} and \citet{La03}.

A structure that has proved fruitful for censored data is the \emph{informative}
(or proportional hazards) model, in which the survival function of the censoring
variable is a power of the survival function of the lifetime. For right
censoring this is the model of \citet{KG76}, studied systematically by
\citet{Cs88, Cs89, CF98, CH81, CM88, GS99, HP89, Pa99, UGC00}. The resulting
power-type estimator of the survival function, often called the ACL
(Abdushukurov--Cheng--Lin) estimator, exploits the informativeness of the
censoring mechanism and is more accurate than the product-limit estimator when
the proportional hazards hypothesis holds \citep{Ab87, Cs88, CF98, Pa99}. Its
kernel-smoothed and density variants, and the bandwidth-selection question they
raise, are treated in \citet{ABM24, CM88, UGC00}. In the two-sided informative
model considered here, \emph{both} censoring laws are powers of the
corresponding survival functions. Such a model was studied by \citet{Ab94} and
extended to competing risks and regression by \citet{Ab98}. Informative
censoring remains an active subject: generalized Koziol--Green models with
partially informative censoring are treated in \citet{Br05}, sieve maximum
likelihood methods under informative censoring in \citet{CHS17}, and the joint
problem of variable selection and estimation in \citet{Liu26}.

The relation of the informative model to \emph{dependent censoring}, the
currently dominant approach, deserves emphasis. There the dependence between
the lifetime and the censoring time is modelled through a copula, and
identifiability is bought either by fixing the copula or by imposing parametric
conditions on the margins \citep{CV23, DV21, DV24, DV25, SKVM25}. The
informative model~\eqref{eq:model} proceeds the other way round: the censoring
laws are tied \emph{functionally} to the lifetime law. This makes the joint
distribution identifiable without any copula and, as Theorem~\ref{th:char}
records, leaves the structural condition \emph{testable} from the data. The
price is a stiffer specification, and it is exactly this trade-off that the
robustness analysis of Section~\ref{sec:robust} quantifies. Estimation from
doubly censored data in general, and the two-sided scheme studied here in
particular, has attracted renewed attention \citep{DGC24, WZH22}; there the
estimators are defined implicitly, through self-consistency or EM iterations,
whereas the estimator proposed below is explicit.

While the semiparametric theory of the two-sided model --- the characterization
theorem and the weak convergence of the associated empirical process --- is by
now well developed (\citealp{Ab94, AM23, Ma20}; see Section~\ref{sec:model}
below), the fully parametric case required in reliability practice has received
much less attention. \citet{AM23} prove the weak convergence of the
semiparametric empirical process (Theorem~\ref{th:wc} below) but go no further
than defining the parametric estimator and illustrating it numerically on a
narrow range of designs. The present paper fills that gap for the single most
important parametric family, the exponential, and supplies the exact Fisher
consistency of the estimating equation, strong consistency, the influence
function and the limit law, the asymptotic bias law of the naive estimator, the
robustness analysis of Section~\ref{sec:robust} and the applications of
Section~\ref{sec:data}. The estimator obtained here is shown to combine
properties rarely found together among procedures designed for doubly censored
data: it is available in closed form, it is exactly Fisher-consistent, it is
$\sqrt n$-consistent with an explicit Gaussian limit at every censoring depth,
and it degrades gracefully under misspecification.

The main results of the paper are the following.
\begin{enumerate}[(i)]
	\item A closed-form pseudo-MLE $\alpha_n$ of the exponential scale under
	two-sided informative censorship is constructed (Section~\ref{ssec:est}). It
	requires no iterative optimization: substituting the empirical df turns the
	fixed-point form of the likelihood equation into an explicit $L$-statistic.
	
	\item The exact Fisher consistency of the estimating equation is proved by
	means of a digamma identity (Theorem~\ref{lem:fisher}); strong consistency is
	established at every censoring depth (Theorem~\ref{th:cons}); and asymptotic
	normality is proved for \emph{all} $\beta,\theta>0$, with an explicit influence
	function whose constants are available in closed form through polygamma
	functions (Theorem~\ref{th:clt} and Corollary~\ref{cor:df}). The key to the
	proof is the identity of Lemma~\ref{lem:phi}, which rewrites the $L$-statistic
	without its unbounded score function $u^{-\lambda}$. The asymptotic variance is
	compared with the Fisher information bound of the model, so that the price of
	passing to closed form is measured (Section~\ref{ssec:eff}): the efficiency
	exceeds $98.6\%$ in the designs considered. It is proved, moreover, that the
	right-censoring half of the model is redundant for this estimator: consistency
	persists whatever the right-censoring law (Theorem~\ref{th:gfree}). The exact
	asymptotic bias of the naive estimator that ignores left censoring is also
	derived (Proposition~\ref{prop:naive}); remarkably, the relative bias
	$(1+\beta)\left[\psi(2+\beta)-\psi(1)\right]-1$ depends on the depth of left
	censoring alone. Finally the scope of the method is delimited
	(Section~\ref{ssec:general}): the likelihood equation is derived for an
	arbitrary family (Proposition~\ref{prop:general}), it admits an explicit
	solution precisely on a proportional hazards class
	(Propositions~\ref{prop:phclass} and~\ref{prop:converse}), and within that
	class all of the above results carry over unchanged.
	
	\item An extensive Monte Carlo study spanning sample sizes from small ($n=30$)
	to large ($n=5000$) and five censoring designs (Section~\ref{sec:mc})
	quantifies the bias, the root-mean-square error (RMSE) and the empirical rate
	of convergence, and documents a systematic efficiency gain over the
	semiparametric power-type estimator.
	
	\item A dedicated robustness analysis (Section~\ref{sec:robust}) examines
	(a)~departures of the lifetime law from exponentiality in the Weibull and gamma
	directions, for which the quasi-maximum-likelihood (pseudo-true parameter)
	apparatus of \citet{Wh82} is used; (b)~violation of the entry condition
	$K=N^{\beta}$, including its replacement by a law of an entirely different
	shape; and (c)~violation of the right-censoring condition
	$1-G=(1-F)^{\theta}$, the finite-sample mirror of Theorem~\ref{th:gfree}. All
	three effects are quantified analytically and numerically, and together they
	expose the asymmetry inside~\eqref{eq:model}: (c) produces no bias, whereas
	(b) does.
	
	\item Three applications to real data (Section~\ref{sec:data}): the classical
	6-MP leukaemia remission data of \citet{Fr63} and \citet{Ge65} illustrate the
	procedure at the boundary case $\beta=0$; the insulating-fluid breakdown data
	of \citet{Ne72} and \citet{La03} do so under artificially imposed two-sided
	informative censorship with known ground truth; and the AIDS data on
	intravenous drug users of \citet{JG11} and \citet{dblcens} --- the one example
	in which left censoring is \emph{genuine} --- demonstrate the practical value
	of Theorem~\ref{th:gfree} and of the extension of
	Section~\ref{ssec:general}.
\end{enumerate}

Section~\ref{sec:model} presents the model and condenses the underlying
asymptotic results; Section~\ref{sec:param} constructs the estimator and
supplies its asymptotic theory, its efficiency and its scope;
Section~\ref{sec:mc} reports the Monte Carlo study, Section~\ref{sec:robust}
the robustness analysis and Section~\ref{sec:data} the applications to real
data; Sections~\ref{sec:disc} and~\ref{sec:concl} contain the discussion and
the conclusions.

\section{The informative model of censorship from both sides}\label{sec:model}

\subsection{Characterization of the model and semiparametric estimation}

Let $\left\{(X_k,L_k,Y_k),\,k\ge1\right\}$ be a sequence of independent copies of
the triple $(X,L,Y)$ and let
\[
S^{(n)}=\left\{(Z_i,\Delta_i),\ i=1,\dots,n\right\}
\]
be the observed sample, where $Z_i=\max\{L_i,\min\{X_i,Y_i\}\}$ and
$\Delta_i=(\delta_i^{(0)},\delta_i^{(1)},\delta_i^{(2)})$ with
$\delta_i^{(0)}=\ind(\min(X_i,Y_i)<L_i)$, $\delta_i^{(1)}=\ind(L_i\le X_i<Y_i)$
and $\delta_i^{(2)}=\ind(L_i\le Y_i<X_i)$, $\ind(A)$ denoting the indicator of
the event $A$. The number of completely observed lifetimes in $S^{(n)}$ is
$\delta_1^{(1)}+\dots+\delta_n^{(1)}$. The statistical problem is to estimate
the df $F$ from $S^{(n)}$, the dfs $K$ and $G$ playing the role of nuisance
parameters. Write $H$ and $N$ for the dfs of $Z_i$ and of
$V_i=\min(X_i,Y_i)$ respectively. Then
\begin{equation}\label{eq:HKN}
	H(x)=N(x)\,K(x),\qquad 1-N(x)=(1-F(x))(1-G(x)),\qquad x\in\Real .
\end{equation}
The \emph{two-sided informative model} postulates the existence of positive
parameters $\theta,\beta$ such that, for every $x\in\Real$,
\begin{equation}\label{eq:model}
	\left\{
	\begin{array}{l}
		1-G(x)=(1-F(x))^{\theta},\\[2pt]
		K(x)=(N(x))^{\beta}.
	\end{array}
	\right.
\end{equation}
Here $\beta$ governs the depth of left censoring and $\theta$ that of right
censoring, values close to zero corresponding to light censoring. The
model~\eqref{eq:model} was studied by \citet{Ab94} and extended to competing
risks by \citet{Ab98}. It contains the Koziol--Green model
\citep{CH81, KG76} as the special case $\beta=0$ (that is $K\equiv1$: no left
censoring). Combining \eqref{eq:HKN} with \eqref{eq:model} gives
\begin{equation}\label{eq:Frepr}
	1-F(x)=\bigl[1-(H(x))^{\lambda}\bigr]^{\gamma},\qquad x\in\Real,
	\qquad
	\lambda=\frac{1}{1+\beta},\quad \gamma=\frac{1}{1+\theta},
\end{equation}
so that values of $\lambda$ and $\gamma$ close to one indicate light censoring.
Representation~\eqref{eq:Frepr} reduces the estimation of $F$ to that of the
triple $(H,\lambda,\gamma)$. Introduce now the subdistribution functions
$T^{(m)}(x)=\Prob(Z_i\le x,\ \delta_i^{(m)}=1)$, $m=0,1,2$, which satisfy
$T^{(0)}+T^{(1)}+T^{(2)}=H$. The following characterization, due to
\citet{Ab94}, is the basic structural property of the model. Beyond its
theoretical interest it makes the specification testable in practice, a fact we
exploit in Sections~\ref{ssec:kviol} and~\ref{sec:data}.

\begin{theorem}[\citealp{Ab94}]\label{th:char}
	The equalities \eqref{eq:model} hold if and only if the random elements $Z_i$
	and $\Delta_i$ are independent.
\end{theorem}

By Theorem~\ref{th:char}, under \eqref{eq:model},
\begin{equation}
	T^{(0)}(x)=(1-\lambda)H(x),\qquad T^{(1)}(x)=\gamma\lambda H(x),\qquad T^{(2)}(x)=(1-\gamma)\lambda H(x),
\end{equation}
and, letting $x\to+\infty$,
\begin{equation}\label{eq:probs}
	p^{(0)}=\Prob(\delta_i^{(0)}=1)=1-\lambda,\qquad
	p^{(1)}=\Prob(\delta_i^{(1)}=1)=\lambda\gamma,\qquad
	p^{(2)}=\lambda(1-\gamma) .
\end{equation}
Estimating the $p^{(m)}$ by the relative frequencies
$p_n^{(m)}=n^{-1}\sum_{i=1}^{n}\delta_i^{(m)}$ yields
\[
\lambda_n=1-p_n^{(0)},\qquad \gamma_n=p_n^{(1)}\bigl(\lambda_n)^{-1},
\]
while estimating $H$ by the empirical df
$H_n(x)=n^{-1}\sum_{i=1}^{n}\ind(Z_i\le x)$ leads, through
\eqref{eq:Frepr}, to the semiparametric power-type (ACL-type) estimator
\begin{equation}\label{eq:Fn}
	F_n(x)=1-\bigl[1-(H_n(x))^{\lambda_n}\bigr]^{\gamma_n},\qquad x\in\Real .
\end{equation}
Under \eqref{eq:model} the endpoints of the supports of $F,G,N,K,H$ coincide:
$T_F=T_G=T_N=T_K=T_H=\inf\{x:H(x)=1\}$ and
$\tau_F=\dots=\tau_H=\sup\{x:H(x)=0\}$.

\subsection{Weak convergence of the semiparametric empirical process}

Consider the normalized process
\begin{equation}
	Q_n(x)=\sqrt n\,\bigl(F_n(x)-F(x)\bigr),\qquad x\in\Real,
\end{equation}
on a compact interval $D=[\tau,T]\subset\Real$ with $\tau_H<\tau\le T<T_H$. On
such a $D$ one has $0<H(\tau)\le H(x)\le H(T)<1$, so that
$(H(x))^{p^{(0)}}-H(x)$ stays bounded away from zero; this alone suffices for
the argument that follows.

\begin{theorem}[\citealp{AM23}]\label{th:wc}
	The sequence $\{Q_n(x),\,x\in D\}$ converges weakly, in the space
	$\ell^{\infty}(D)$ of bounded functions on $D$ equipped with the sup-norm, to a
	centred Gaussian process $\{A(x),\,x\in D\}$ whose covariance is, for
	$x_1,x_2\in D$,
	\begin{align*}
		\cov\{A(x_1),A(x_2)\}&=(1-F(x_1))(1-F(x_2))\times\\
		&\quad\times\bigl\{a(x_1)a(x_2)\left[H(\min(x_1,x_2))-H(x_1)H(x_2)\right]\\
		&\qquad\quad+b(x_1)b(x_2)\,p^{(0)}(1-p^{(0)})\\
		&\qquad\quad+c(x_1)c(x_2)\,p^{(1)}(1-p^{(1)})\\
		&\qquad\quad-\bigl[b(x_1)c(x_2)+b(x_2)c(x_1)\bigr]p^{(0)}p^{(1)}\bigr\},
	\end{align*}
	where
	\begin{align*}
		a(x)&=\frac{p^{(1)}}{(H(x))^{p^{(0)}}-H(x)},\\
		b(x)&=-\frac{a(x)}{1-p^{(0)}}\,H(x)\log H(x) + \frac{p^{(1)}}{1-p^{(0)}}\,c(x),\\
		c(x)&=-\frac{1}{1-p^{(0)}}\,\log\bigl[1-(H(x))^{1-p^{(0)}}\bigr].
	\end{align*}
\end{theorem}

\begin{proofsketch}
	The result rests on the asymptotic representation
	\begin{equation}\label{eq:repr}
		F_n(x)-F(x)=n^{-1}\sum_{i=1}^{n}\Psi_x(Z_i,\Delta_i)+r_n(x),\qquad
		\sup_{x\in D}|r_n(x)|\overset{a.s.}{=}{\rm O}\left(n^{-1}\log n\right),
	\end{equation}
	in which
	\begin{align*}
		\Psi_x(Z_i,\Delta_i)&=(1-F(x))\,a(x)\bigl(\ind(Z_i\le x)-H(x)\bigr)+(1-F(x))\,b(x)\bigl(\delta_i^{(0)}-p^{(0)}\bigr)\\
		&\quad+(1-F(x))\,c(x)\bigl(\delta_i^{(1)}-p^{(1)}\bigr).
	\end{align*}
	Representation~\eqref{eq:repr} follows from a second-order Taylor expansion of
	the smooth map $(u,y,z)\mapsto(1-u^{1-y})^{z/(1-y)}$ about the point
	$(H(x),p^{(0)},p^{(1)})$, evaluated at $(H_n(x),p_n^{(0)},p_n^{(1)})$; the
	quadratic remainder is bounded on $D$ by the
	Dvoretzky--Kiefer--Wolfowitz--Massart inequality \citep{Mss90} applied to $H_n$
	and to the frequencies $p_n^{(m)}$, the compactness of $D$ keeping every
	denominator away from zero. Since the summands $\Psi_x$ are bounded, centred and
	independent and identically distributed, convergence of the
	finite-dimensional distributions follows from the multivariate central limit
	theorem \citep{VW96}, and tightness is checked by the standard fourth-moment
	bound together with the moment criterion of \citet{VW96}. The covariance
	formula is a direct computation that uses the independence of $Z_i$ and
	$\Delta_i$ guaranteed by Theorem~\ref{th:char}: that independence annihilates
	the cross terms between $a$ and $b,c$, while the last term comes from
	$\Ex[(\delta^{(0)}-p^{(0)})(\delta^{(1)}-p^{(1)})]=-p^{(0)}p^{(1)}$, which
	arises because $\delta^{(0)}\delta^{(1)}=0$. A complete proof is given in
	\citet{AM23}; it is restated here to keep the paper self-contained and for the
	comparison made in Corollary~\ref{cor:df}. As a check, at $\beta=0$ one has
	$\delta^{(0)}\equiv0$, the term $b(x)$ vanishes and the covariance above
	reduces exactly to the formula given for the one-sided Koziol--Green model by
	\citet{CL87}.
\end{proofsketch}

Theorem~\ref{th:wc} sets the semiparametric benchmark against which any
parametric competitor is to be judged: on $D$ the estimator \eqref{eq:Fn}
converges at rate $\sqrt n$ and its limit law is a fully explicit Gaussian
process. We now turn to the central object of the paper.

\section{Parametric estimation of the exponential distribution}\label{sec:param}

\subsection{The closed-form pseudo-MLE}\label{ssec:est}

Suppose the lifetime of the observed unit is exponentially distributed,
\[
X_i\sim F(x,\alpha)=1-e^{-x/\alpha},\quad x\ge0,\ \alpha>0,
\]
and that the censoring laws obey \eqref{eq:model}. Then
$Y_i\sim G(x,\alpha)=1-(1-F(x,\alpha))^{\theta}=1-e^{-\theta x/\alpha}$ and
$L_i\sim K(x,\alpha)=(N(x,\alpha))^{\beta}$ with
$N(x,\alpha)=1-e^{-(1+\theta)x/\alpha}$. Together with \eqref{eq:Frepr} this
gives $1-e^{-x/\alpha}=1-[1-(H(x))^{\lambda}]^{\gamma}$, whence
\begin{equation}\label{eq:Hexp}
	H(x)=H(x;\alpha)=\Bigl(1-e^{-\frac{x}{\alpha\gamma}}\Bigr)^{1/\lambda},
	\qquad
	h(x;\alpha)=\frac{1}{\alpha\lambda\gamma}\,e^{-\frac{x}{\alpha\gamma}}\Bigl(1-e^{-\frac{x}{\alpha\gamma}}\Bigr)^{\frac{1}{\lambda}-1} .
\end{equation}
Treating $(\lambda,\gamma)$ as known for the moment, the log-likelihood of the
observed times $Z_1,\dots,Z_n$ is
\begin{equation}
	\ell_n(\alpha)=\sum_{i=1}^{n}\log h(Z_i;\alpha)
	=-n\log(\alpha\lambda\gamma)-\frac{1}{\alpha\gamma}\sum_{i=1}^{n}Z_i
	+\frac{1-\lambda}{\lambda}\sum_{i=1}^{n}\log\Bigl(1-e^{-\frac{Z_i}{\alpha\gamma}}\Bigr),
\end{equation}
with derivative
\[
\frac{\partial\ell_n(\alpha)}{\partial\alpha}
=-\frac{n}{\alpha}+\frac{1}{\alpha^{2}\gamma}\sum_{i=1}^{n}Z_i
-\frac{1-\lambda}{\lambda}\,\frac{1}{\alpha^{2}\gamma}\sum_{i=1}^{n}
\frac{Z_i\,e^{-Z_i/(\alpha\gamma)}}{1-e^{-Z_i/(\alpha\gamma)}} .
\]
Using the identity $1-e^{-z/(\alpha\gamma)}=(H(z;\alpha))^{\lambda}$, the
likelihood equation $\partial\ell_n/\partial\alpha=0$ takes the fixed-point form
\begin{equation}\label{eq:fixedpoint}
	\alpha=\frac{1}{n\gamma\lambda}\sum_{i=1}^{n}Z_i\Bigl(1-\frac{1-\lambda}{(H(Z_i;\alpha))^{\lambda}}\Bigr).
\end{equation}
Replacing the unknown triple $(H(\cdot\,;\alpha),\lambda,\gamma)$ by its
empirical counterpart $(H_n,\lambda_n,\gamma_n)$ removes the dependence of the
right-hand side on $\alpha$ and yields the \emph{closed-form} pseudo-MLE
\begin{equation}\label{eq:alphan}
	\alpha_n=\frac{1}{\lambda_n\gamma_n n}\sum_{i=1}^{n}Z_i
	\Bigl(1-\frac{1-\lambda_n}{(H_n(Z_i))^{\lambda_n}}\Bigr),
\end{equation}
and with it the estimator $F(x,\alpha_n)=1-e^{-x/\alpha_n}$ of the lifetime df.
Writing $Z_{(1)}\le\dots\le Z_{(n)}$ for the order statistics of the sample
(ties have probability zero), one has $H_n(Z_{(i)})=i/n$ and therefore
\begin{equation}\label{eq:Lstat}
	\alpha_n=\frac{1}{\lambda_n\gamma_n}\Bigl[\bar Z_n-(1-\lambda_n)\,\frac1n\sum_{i=1}^{n}\Bigl(\frac{i}{n}\Bigr)^{-\lambda_n}Z_{(i)}\Bigr],
\end{equation}
an $L$-statistic with score function $J(u)=u^{-\lambda_n}$, corrected by the
multinomial frequencies. The estimator \eqref{eq:alphan} is computed directly
from the ordered sample: no starting values, no iterations and no numerical
optimization are required. For comparison, nonparametric maximum likelihood
estimators for doubly censored data call for iterative self-consistency
algorithms, and parametric likelihoods in the non-informative two-sided scheme
call for numerical maximization; \eqref{eq:alphan} needs neither.

\begin{remark}
	The weights entering \eqref{eq:alphan} are automatically stable at the lower
	edge of the sample: by \eqref{eq:Hexp}, $z\,(H(z))^{-\lambda}\to\alpha\gamma$ as
	$z\downarrow0$, and correspondingly $n^{\lambda_n}Z_{(1)}={\rm O}_{\Prob}(1)$, so
	that no truncation or trimming near the origin is needed.
\end{remark}

\begin{remark}\label{rem:KG}
	When $\beta=0$ (no left censoring) one has $\lambda_n\equiv1$, the correction
	$(1-\lambda_n)H_n^{-\lambda_n}$ in \eqref{eq:alphan} vanishes and, since
	$\gamma_n=p_n^{(1)}$, the estimator reduces to
	\[
	\alpha_n=\frac{1}{\gamma_n n}\sum_{i=1}^{n}Z_i
	=\sum_{i=1}^{n}Z_i\Big/\sum_{i=1}^{n}\delta_i^{(1)} ,
	\]
	the total-time-on-test (TTT) estimator. \citet{ES53} derived it for
	\emph{type~II} censoring, in which $n$ items are put on test and the experiment
	is stopped after the $r$th failure, their estimator being
	$\widehat\theta_{r,n}=\bigl[x_{1,n}+\dots+x_{r,n}+(n-r)x_{r,n}\bigr]/r$, the
	ratio of the total time on test to the number of failures. Under arbitrary
	right censoring the same ratio remains exactly the maximum likelihood estimator
	of the exponential scale.
	
	The simplicity of the estimator is a consequence of the \emph{constant hazard}
	of the exponential distribution: because $f/(1-F)\equiv1/\alpha$, the
	log-likelihood becomes
	$-\bigl(\sum_i\delta_i^{(1)}\bigr)\log\alpha-\alpha^{-1}\sum_i Z_i$ and is
	solved explicitly. This simplicity is not lost entirely: on the class
	$1-F(x;\alpha)=\exp\{-A(x)/\alpha\}$ of Proposition~\ref{prop:phclass} the same
	argument gives $\sum_i A(Z_i)\big/\sum_i\delta_i^{(1)}$, again exactly the
	maximum likelihood estimator. In this sense the TTT estimator is nothing but
	the case $A(x)=x$ of the estimator studied here.
	
	For the one-sided Koziol--Green model, \citet{CL87} constructed
	$\widehat S_F=S_n^{\alpha_n}$ as the maximum likelihood estimator, $S_n$ being
	the empirical survival function of the observed time and $\alpha_n$ the
	proportion of uncensored observations. That is the case $\beta=0$ of the
	\emph{semiparametric} estimator \eqref{eq:Fn}, not of the \emph{parametric}
	estimator \eqref{eq:alphan}: at $\lambda_n=1$ formula \eqref{eq:Fn} becomes
	exactly $1-S_n^{\gamma_n}$ with $\gamma_n=p_n^{(1)}=\alpha_n$. Thus
	\eqref{eq:alphan} is a genuine two-sided extension of the classical theory and
	is distinct from the semiparametric line of work. The derivation itself does
	not depend on exponentiality: in Section~\ref{ssec:general} the likelihood
	equation is written for an arbitrary family and the class on which it admits a
	closed solution is identified. We first treat the exponential case in full.
\end{remark}

\subsection{Exact Fisher consistency}

The population version of the right-hand side of \eqref{eq:alphan} is the
statistical functional
\begin{equation}
	T(\widetilde H,l,g)=\frac{1}{l\,g}\int_{0}^{\infty}z\Bigl(1-\frac{1-l}{(\widetilde H(z))^{l}}\Bigr)\,d\widetilde H(z),
\end{equation}
$\widetilde H$ being an arbitrary df.

The following theorem shows that the estimating equation defining $\alpha_n$ is
Fisher-consistent \emph{exactly}, not merely asymptotically, and for every
censoring configuration of the model.

\begin{theorem}\label{lem:fisher}
	Under the model \eqref{eq:model} with exponentially distributed lifetime,
	\[
	\Ex\,Z=\alpha\gamma\bigl[\psi(1+1/\lambda)-\psi(1)\bigr],
	\qquad
	\Ex\bigl[Z (H(Z))^{-\lambda}\bigr]=\frac{\alpha\gamma}{1-\lambda}\bigl[\psi(1/\lambda)-\psi(1)\bigr],
	\]
	and consequently $T(H(\cdot\,;\alpha),\lambda,\gamma)=\alpha$ exactly, where
	$\psi$ denotes the digamma function.
\end{theorem}

\begin{proof}
	Substitute $u=1-e^{-x/(\alpha\gamma)}$, so that $H=u^{1/\lambda}$ and
	$dx=\alpha\gamma\,du/(1-u)$. By the classical integral representation of the
	digamma function,
	\[
	\Ex\,Z=\int_{0}^{\infty}(1-H(x))\,dx
	=\alpha\gamma\int_{0}^{1}\frac{1-u^{1/\lambda}}{1-u}\,du
	=\alpha\gamma\bigl[\psi(1+1/\lambda)-\psi(1)\bigr].
	\]
	Similarly, with $v=(H(z))^{\lambda}$,
	\[
	\Ex\bigl[Z H(Z)^{-\lambda}\bigr]
	=\frac{\alpha\gamma}{\lambda}\int_{0}^{1}\bigl(-\log(1-v)\bigr)v^{1/\lambda-2}\,dv
	=\frac{\alpha\gamma\,[\psi(1/\lambda)-\psi(1)]}{1-\lambda},
	\]
	where the formula $\int_0^1(-\log(1-v))v^{a-1}dv=[\psi(a+1)-\psi(1)]/a$ with
	$a=(1-\lambda)/\lambda$ has been used. Hence, by the recurrence
	$\psi(x+1)-\psi(x)=1/x$ at $x=1/\lambda$,
	\[
	T(H,\lambda,\gamma)
	=\frac{1}{\lambda\gamma}\Bigl(\Ex Z-(1-\lambda)\,\Ex\bigl[Z H^{-\lambda}\bigr]\Bigr)
	=\frac{\alpha}{\lambda}\bigl[\psi(1+1/\lambda)-\psi(1/\lambda)\bigr]=\alpha. \qed
	\]
\end{proof}

\begin{remark}
	In \eqref{eq:model} one has $\beta>0$, hence $0<\lambda<1$, and both
	expressions above are meaningful. In the boundary case $\beta=0$
	($\lambda=1$, Remark~\ref{rem:KG}) the right-hand side of the second identity
	takes the form $0/0$. The singularity is removable and the limit equals
	$\alpha\gamma\,\psi'(1)=\alpha\gamma\pi^{2}/6$, since
	$[\psi(1/\lambda)-\psi(1)]/(1-\lambda)\to\psi'(1)$ as $\lambda\uparrow1$.
	Fisher consistency itself does not break down at $\lambda=1$: in the expression
	for $T$ the factor $(1-\lambda)$ multiplying $\Ex[Z H^{-\lambda}]$ cancels the
	$(1-\lambda)$ in the denominator, so that $T(H,\lambda,\gamma)=\alpha$ holds for
	all $0<\lambda\le1$. At $\lambda=1$ this is also seen directly: the second term
	vanishes and $T=\Ex Z/\gamma=\alpha[\psi(2)-\psi(1)]=\alpha$.
\end{remark}

\subsection{Asymptotic properties}

\begin{theorem}\label{th:cons}
	Under \eqref{eq:model} with exponential lifetime, $\alpha_n\to\alpha$ almost
	surely as $n\to\infty$, for all $\beta,\theta>0$.
\end{theorem}

\begin{proofsketch}
	By the strong law of large numbers $\lambda_n\to\lambda$, $\gamma_n\to\gamma$
	and $\bar Z_n\to\Ex Z$ almost surely. For the $L$-statistic part of
	\eqref{eq:Lstat} the score--quantile product is
	$u^{-\lambda}H^{-1}(u)=-\alpha\gamma\,u^{-\lambda}\log(1-u^{\lambda})$, which is
	continuous on $(0,1)$, tends to $\alpha\gamma$ as $u\downarrow0$ and diverges
	only logarithmically as $u\uparrow1$. In particular it is integrable by
	Theorem~\ref{lem:fisher}: $\int_0^1u^{-\lambda}H^{-1}(u)\,du=\kappa_1<\infty$.
	Under that integrability condition the strong law of large numbers for linear
	functions of order statistics (\citealp{SW86}, Chapter~19) gives
	$n^{-1}\sum_i (i/n)^{-\lambda_n}Z_{(i)}\to\kappa_1$ almost surely; the random
	exponent $\lambda_n$ is handled by an almost sure sandwich argument between
	$u^{-\lambda\pm\varepsilon}$. The assertion then follows from the continuous
	mapping theorem and Theorem~\ref{lem:fisher}.
\end{proofsketch}

The next result supplies the limit distribution. For $z>0$ set
\begin{equation}\label{eq:ellz}
	\ell(z)=\int_{0}^{\infty}(H(x))^{-\lambda}\bigl(H(x)-\ind(z\le x)\bigr)\,dx .
\end{equation}
This function is available in closed form.

\begin{lemma}\label{lem:ell}
	For all $\lambda\in(0,1)$ and $z>0$,
	\begin{equation}\label{eq:ellclosed}
		\ell(z)=z+\alpha\gamma\log\Bigl(1-e^{-\frac{z}{\alpha\gamma}}\Bigr)
		-\alpha\gamma\bigl[\psi(1/\lambda)-\psi(1)\bigr].
	\end{equation}
	In particular $\Ex\,\ell(Z)=0$ exactly, and $\Ex\,\ell^{2}(Z)<\infty$.
\end{lemma}

\begin{proof}
	Split at $z$:
	$\ell(z)=\int_0^z H^{1-\lambda}dx-\int_z^{\infty}H^{-\lambda}(1-H)\,dx$.
	Substituting first $u=H(x)$, $dx=\alpha\gamma\lambda u^{\lambda-1}(1-u^{\lambda})^{-1}du$,
	and then $s=u^{\lambda}$, both integrals simplify and, with
	$s_0=(H(z))^{\lambda}=1-e^{-z/(\alpha\gamma)}$,
	\[
	\frac{\ell(z)}{\alpha\gamma}=\int_{0}^{1}\frac{s^{1/\lambda-1}-1}{1-s}\,ds
	-\log(1-s_0)+\log s_0 .
	\]
	The first integral equals $-[\psi(1/\lambda)-\psi(1)]$ by the Gauss
	representation of the digamma function, while the second term is
	$-\log(1-s_0)=z/(\alpha\gamma)$. Collecting the pieces gives
	\eqref{eq:ellclosed}. To verify centring, note that $U=H(Z)$ is uniform and
	$\log(1-e^{-Z/(\alpha\gamma)})=\lambda\log U$, so that
	$\Ex\log(1-e^{-Z/(\alpha\gamma)})=-\lambda$. Together with
	$\Ex Z=\alpha\gamma[\psi(1+1/\lambda)-\psi(1)]$ from Theorem~\ref{lem:fisher},
	\[
	\Ex\,\ell(Z)=\alpha\gamma\bigl[\psi(1+1/\lambda)-\psi(1/\lambda)\bigr]-\alpha\gamma\lambda
	=\alpha\gamma\lambda-\alpha\gamma\lambda=0,
	\]
	because $\psi(x+1)-\psi(x)=1/x$ at $x=1/\lambda$. Square integrability follows
	from \eqref{eq:ellasym} below together with the exponential tail of $H$.
\end{proof}

It is immediate from \eqref{eq:ellclosed} that $\ell$ is \emph{not} bounded,
contrary to what a hasty reading of the form \eqref{eq:ellz} might suggest:
\begin{equation}\label{eq:ellasym}
	\ell(z)=\alpha\gamma\log z+{\rm O}(1)\quad (z\downarrow0),
	\qquad
	\ell(z)=z+{\rm O}(1)\quad (z\to\infty),
\end{equation}
so that $\ell$ has a logarithmic singularity at the origin and grows linearly at
infinity. Both rates are needed for the argument that $\ell$ is square
integrable with respect to $H$. Writing
$\kappa_1=\Ex\bigl[Z(H(Z))^{-\lambda}\bigr]$ and
$\kappa_2=\Ex\bigl[Z(H(Z))^{-\lambda}\log H(Z)\bigr]$, we introduce the constant
\begin{equation}\label{eq:clam}
	c_{\lambda}=\frac{\partial T(H,l,\gamma)}{\partial l}\Big|_{l=\lambda}
	=-\frac{\alpha}{\lambda}+\frac{\kappa_1+(1-\lambda)\kappa_2}{\lambda\gamma} = \frac{\alpha}{\lambda^2}\big(\psi'(\tfrac{1}{\lambda})-\lambda\big),
\end{equation}
where, besides $\kappa_1=\alpha\gamma[\psi(1/\lambda)-\psi(1)]/(1-\lambda)$ from
Theorem~\ref{lem:fisher},
\[
\kappa_2=\frac{\alpha\gamma\,\psi'(1/\lambda)}{\lambda(1-\lambda)}
-\frac{\alpha\gamma\,[\psi(1/\lambda)-\psi(1)]}{(1-\lambda)^{2}},
\]
$\psi'$ being the trigamma function. All the constants above are thus available
in closed form.

The key to the proof is the following identity, which carries the
$L$-statistic from the quantile domain to the distribution-function domain.

\begin{lemma}\label{lem:phi}
	Put $\Phi(t)=\dfrac{1-t^{1-\lambda}}{1-\lambda}$, $t\in[0,1]$ (since
	$\lambda<1$ there is no singularity at $t=0$: $0^{1-\lambda}=0$). For every
	continuous df $\widetilde H$ supported on $[0,\infty)$,
	\begin{equation}\label{eq:phiident}
		T_L(\widetilde H):=\int_{0}^{1}u^{-\lambda}\widetilde H^{-1}(u)\,du
		=\int_{0}^{\infty}\Phi\bigl(\widetilde H(x)\bigr)\,dx ,
	\end{equation}
	both sides being finite or infinite together. Here
	$\Phi'(t)=-t^{-\lambda}$ and $\Phi''(t)=\lambda t^{-\lambda-1}$ on $(0,1)$;
	these derivatives are unbounded as $t\downarrow0$, but $\Phi$ \emph{itself} is
	continuous and \emph{bounded} on the whole of $[0,1]$, with
	$\Phi(0)=(1-\lambda)^{-1}$, $\Phi(1)=0$ and
	$|\Phi(t)-\Phi(s)|\le|t-s|^{1-\lambda}/(1-\lambda)$ for all $s,t\in[0,1]$
	(with equality at $s=0$).
\end{lemma}

\begin{proof}
	Apply Tonelli's theorem, using the equivalence
	$x<\widetilde H^{-1}(u)\iff\widetilde H(x)<u$; the integrand
	$u^{-\lambda}\ind\bigl(x<\widetilde H^{-1}(u)\bigr)$ is non-negative, so the
	identity holds in $[0,\infty]$. Thus
	\begin{align*}
		\int_0^1 u^{-\lambda}\widetilde H^{-1}(u)\,du
		&=\int_0^1 u^{-\lambda}\!\!\int_0^{\infty}\!\!\ind\bigl(x<\widetilde H^{-1}(u)\bigr)dx\,du\\
		&=\int_0^{\infty}\!\!\int_{\widetilde H(x)}^{1}\!\!u^{-\lambda}du\,dx
		=\int_0^{\infty}\!\Phi\bigl(\widetilde H(x)\bigr)dx .
	\end{align*}
	The inner integral is finite even when the lower limit is $\widetilde H(x)=0$,
	that is when $x$ lies to the left of the support of $\widetilde H$: because
	$\lambda<1$ it then equals exactly
	$\int_0^1u^{-\lambda}du=\Phi(0)=(1-\lambda)^{-1}$. The set of such points is
	$[0,\tau)$ with $\tau=\inf\{x:\widetilde H(x)>0\}$ and contributes
	$\tau\,\Phi(0)$ to the right-hand side, matched on the left by the same term
	arising from $\widetilde H^{-1}(u)\ge\tau$. The H\"older bound follows from the
	subadditivity of $t\mapsto t^{1-\lambda}$.
\end{proof}

The point of \eqref{eq:phiident} is that on its right-hand side the score
function $J(u)=u^{-\lambda}$ has been integrated away: what remains is $\Phi$
alone, which is bounded on $[0,1]$. The unboundedness of $J$ --- the sole
reason for invoking the classical central limit theorems for $L$-statistics with
smooth weights \citep{Mas81, St74}, and hence for the restriction
$\lambda>1/2$ --- is therefore an artefact of writing the statistic in the
quantile domain. In the distribution-function domain the only thing left to
control is a second-order Taylor remainder. Transferring an $L$-functional from
the quantile domain to the distribution-function domain is itself a classical
device in the theory of $L$-statistics (\citealp{SW86}, Chapter~19); the
novelty here lies not in the device but in its consequence, namely that the
restriction on $\lambda$ disappears altogether. The next lemma carries this out.

\begin{lemma}\label{lem:lin}
	For all $\lambda\in(0,1)$,
	\[
	R_n:=T_L(H_n)-T_L(H)-\frac1n\sum_{i=1}^{n}\ell(Z_i)={\rm o}_{\Prob}\bigl(n^{-1/2}\bigr).
	\]
\end{lemma}

\begin{proof}
	Let $U_i=H(Z_i)$ be uniform with empirical df $\Gamma_n$; since $H$ is
	continuous, $H_n(x)=\Gamma_n(H(x))$. Put $D_n=H_n-H$, so that
	$\Ex\,D_n^2(x)=n^{-1}H(x)(1-H(x))$ exactly. By \eqref{eq:Hexp} one has $H(x)>0$
	for every $x>0$, hence the quantities $\Phi'(H(x))$ and $\Phi''(H(x))$ appearing
	below are finite (the derivatives in Lemma~\ref{lem:phi} are defined on
	$(0,1)$; the only point excluded, $x=0$, is Lebesgue-null). By
	Lemma~\ref{lem:phi} and the identity
	$\int_0^{\infty}\Phi'(H)D_n\,dx=-n^{-1}\sum_i\ell(Z_i)$,
	\[
	R_n=\int_{0}^{\infty}\Bigl[\Phi\bigl(H_n(x)\bigr)-\Phi\bigl(H(x)\bigr)
	-\Phi'\bigl(H(x)\bigr)D_n(x)\Bigr]dx .
	\]
	Since $\lambda<1$ we have $\lambda/(1+\lambda)<1/2$, so there exists $\rho$ with
	\begin{equation}\label{eq:rho}
		\frac{\lambda}{1+\lambda}<\rho<\frac12 .
	\end{equation}
	Set $a_n=n^{-\rho}$ and $x_n=H^{-1}(a_n)$, and split the integral at $x_n$.
	Throughout, the substitution $u=H(x)$ gives
	$dx=\alpha\gamma\lambda u^{\lambda-1}(1-u^{\lambda})^{-1}du$.
	
	\emph{(a) Outer region $x>x_n$.} Introduce the event
	$E_n=\{\inf_{u\ge a_n}\Gamma_n(u)/u\ge\tfrac12\}$. By the
	Dvoretzky--Kiefer--Wolfowitz--Massart inequality \citep{Mss90},
	$\Prob(E_n^{c})\le\Prob(\|\Gamma_n-\mathrm{id}\|_{\infty}>a_n/2)\le2e^{-na_n^{2}/2}\to0$
	because $\rho<1/2$. On $E_n$, every $v$ between $H(x)$ and $H_n(x)$ satisfies
	$v\ge H(x)/2$, whence $\Phi''(v)\le\lambda2^{\lambda+1}(H(x))^{-\lambda-1}$ and,
	by Taylor's formula,
	\[
	\Bigl|\int_{x_n}^{\infty}[\;\cdot\;]\,dx\Bigr|
	\le\lambda2^{\lambda}\int_{x_n}^{\infty}D_n^{2}(x)\,(H(x))^{-\lambda-1}dx .
	\]
	Taking expectations and substituting $u=H(x)$,
	\begin{align*}
		\Ex\int_{x_n}^{\infty}\!\!D_n^{2}H^{-\lambda-1}dx
		&=\frac1n\int_{a_n}^{1}\!u(1-u)\,u^{-\lambda-1}
		\frac{\alpha\gamma\lambda u^{\lambda-1}}{1-u^{\lambda}}\,du\\
		&=\frac{\alpha\gamma\lambda}{n}\int_{a_n}^{1}\frac{1-u}{u\,(1-u^{\lambda})}\,du
		={\rm O}\!\Bigl(\frac{\log n}{n}\Bigr),
	\end{align*}
	since the integrand behaves like $u^{-1}$ at $0$ and like $\lambda^{-1}$ at $1$.
	Note that $\lambda$ \emph{cancels exactly} here:
	$u^{-\lambda-1}\cdot u^{\lambda-1}=u^{-2}$, and no trace of $\lambda$ remains.
	
	\emph{(b) Inner region, linear term.} By the Cauchy--Schwarz inequality,
	\begin{align*}
		\Ex\int_{0}^{x_n}\!\!H^{-\lambda}|D_n|\,dx
		&\le n^{-1/2}\!\!\int_{0}^{a_n}\!\!u^{-\lambda}\bigl(u(1-u)\bigr)^{1/2}
		\frac{\alpha\gamma\lambda u^{\lambda-1}}{1-u^{\lambda}}du\\
		&\le C n^{-1/2}\!\!\int_{0}^{a_n}\!\!u^{-1/2}du=C'n^{-1/2}a_n^{1/2},
	\end{align*}
	which is ${\rm o}(n^{-1/2})$ because $a_n\to0$. (Again $\lambda$ cancels.)
	
	\emph{(c) Inner region, nonlinear term.} By the H\"older bound of
	Lemma~\ref{lem:phi} and Jensen's inequality ($t\mapsto t^{(1-\lambda)/2}$ being
	concave), $\Ex|D_n(x)|^{1-\lambda}\le\bigl(n^{-1}u(1-u)\bigr)^{(1-\lambda)/2}$,
	so that
	\begin{align*}
		\Ex\int_{0}^{x_n}\!\bigl|\Phi(H_n)-\Phi(H)\bigr|dx
		&\le\frac{n^{-\frac{1-\lambda}{2}}}{1-\lambda}\int_{0}^{a_n}\!\!u^{\frac{1-\lambda}{2}}
		\frac{\alpha\gamma\lambda u^{\lambda-1}}{1-u^{\lambda}}du\\
		&\le C n^{-\frac{1-\lambda}{2}}\!\!\int_{0}^{a_n}\!\!u^{-\frac{1-\lambda}{2}}du
		=C'n^{-\frac{1-\lambda}{2}}a_n^{\frac{1+\lambda}{2}},
	\end{align*}
	convergence of the integral being ensured by $-\tfrac{1-\lambda}{2}>-1$. With
	$a_n=n^{-\rho}$ this is ${\rm O}(n^{-\eta})$ with
	$\eta=\tfrac{1-\lambda}{2}+\rho\tfrac{1+\lambda}{2}$, and the requirement
	$\eta>\tfrac12$ is precisely $\rho>\lambda/(1+\lambda)$, that is
	\eqref{eq:rho}.
	
	Combining (a)--(c) yields $R_n={\rm o}_{\Prob}(n^{-1/2})$.
\end{proof}

\begin{remark}
	The proof used only the exact identity $\Ex D_n^2=n^{-1}H(1-H)$, the
	inequalities of Jensen and Cauchy--Schwarz, Tonelli's theorem and the
	Dvoretzky--Kiefer--Wolfowitz--Massart inequality; no theory of weighted
	empirical processes was needed, and no condition was imposed on $\lambda$
	beyond $\lambda<1$. Choosing $a_n$ by a dyadic Bernstein bound instead of
	DKWM admits any $\rho<1$ and strengthens the conclusion to
	$R_n={\rm O}_{\Prob}(n^{-1+\varepsilon})$ for every $\varepsilon>0$.
\end{remark}

\begin{theorem}\label{th:clt}
	Let $\beta>0$ and $\theta>0$. Then
	\[
	\sqrt n\,(\alpha_n-\alpha)\ \xrightarrow{d}\ N\bigl(0,\sigma^{2}\bigr),
	\qquad
	\sigma^{2}=\Var\,\psi_{\alpha}(Z,\Delta),
	\]
	where the influence function is
	\begin{equation}\label{eq:IF}
		\begin{split}
			\psi_{\alpha}(z,\delta^{(0)},\delta^{(1)})
			&=\frac{1}{\lambda\gamma}\Bigl[(z-\Ex Z)-(1-\lambda)\,\ell(z)\Bigr]
			-c_{\lambda}\bigl(\delta^{(0)}-p^{(0)}\bigr)\\
			&\quad-\frac{\alpha}{\lambda\gamma}\Bigl[\bigl(\delta^{(1)}-p^{(1)}\bigr)+\gamma\bigl(\delta^{(0)}-p^{(0)}\bigr)\Bigr].
		\end{split}
	\end{equation}
\end{theorem}

\begin{proof}
	Since $\lambda_n=1-p_n^{(0)}$ and $\gamma_n=p_n^{(1)}/\lambda_n$, we have
	$\lambda_n\gamma_n=p_n^{(1)}$, so that \eqref{eq:Lstat} can be written as
	\begin{equation}\label{eq:alphaclean}
		\alpha_n=\frac{\bar Z_n-p_n^{(0)}\,S_n\bigl(1-p_n^{(0)}\bigr)}{p_n^{(1)}},
		\qquad
		S_n(l)=\frac1n\sum_{i=1}^{n}\Bigl(\frac in\Bigr)^{-l}Z_{(i)} .
	\end{equation}
	By Theorem~\ref{th:char} the vector $\Delta_i$ is independent of $Z_i$, hence
	$(p_n^{(0)},p_n^{(1)})$ is independent of the whole $Z$-sample. The proof is in
	four steps.
	
	\emph{Step 1: discretization.} By Lemma~\ref{lem:phi},
	$T_L(H_n)=\sum_j w_jZ_{(j)}$ with $w_j=\int_{(j-1)/n}^{j/n}u^{-\lambda}du$,
	whereas the estimator involves
	$S_n(\lambda)=\sum_j n^{-1}(j/n)^{-\lambda}Z_{(j)}$. As $u^{-\lambda}$ is
	decreasing, $e_j:=w_j-n^{-1}(j/n)^{-\lambda}\ge0$, and by the mean value theorem
	$e_j\le\tfrac{\lambda}{2}n^{-2}((j-1)/n)^{-\lambda-1}$ for $j\ge2$, while
	$e_1=\tfrac{\lambda}{1-\lambda}n^{\lambda-1}$. Put $u_0=2^{-1/\lambda}$: for
	$u\le u_0$ one has $u^{\lambda}\le\tfrac12$, hence
	$-\log(1-u^{\lambda})\le u^{\lambda}/(1-u^{\lambda})\le2u^{\lambda}$, that is
	$H^{-1}(u)\le2\alpha\gamma u^{\lambda}$; moreover
	$\Ex U_{(j)}^{\lambda}\asymp(j/n)^{\lambda}$. Split the sum at $j/n=u_0$. On the
	lower part,
	\begin{align*}
		\Ex\Bigl|\sum_{j/n\le u_0} e_jZ_{(j)}\Bigr|
		&\lesssim n^{\lambda-1}\Ex Z_{(1)}+\frac1{n^{2}}\sum_{j\ge2}\Bigl(\frac jn\Bigr)^{-\lambda-1}\!\!\Ex Z_{(j)}\\
		&\lesssim n^{-1}+\frac1{n^{2}}\sum_{j\ge2}\frac nj={\rm O}\Bigl(\frac{\log n}{n}\Bigr).
	\end{align*}
	On the upper part ($j/n>u_0$) one has $((j-1)/n)^{-\lambda-1}={\rm O}(1)$ and
	$\Ex Z_{(j)}={\rm O}(\log n)$, so that this term is likewise of order
	$n^{-2}\cdot n\cdot{\rm O}(\log n)={\rm O}(n^{-1}\log n)$. (Here $u_0$ is a
	constant depending on $\lambda$ only, not on $n$.) Hence
	$S_n(\lambda)-T_L(H_n)={\rm O}_{\Prob}(n^{-1}\log n)$ and, combined with
	Lemma~\ref{lem:lin},
	\begin{equation}\label{eq:Slin}
		S_n(\lambda)=\kappa_1+\frac1n\sum_{i=1}^{n}\ell(Z_i)+{\rm o}_{\Prob}(n^{-1/2}).
	\end{equation}
	
	\emph{Step 2: the random exponent.} Let
	\[
	S(l)=\Ex\bigl[Z(H(Z))^{-l}\bigr]
	=\frac{\alpha\gamma}{1-l}\Bigl[\psi\Bigl(1+\tfrac{1-l}{\lambda}\Bigr)-\psi(1)\Bigr]
	\]
	be the population version; in particular $S(\lambda)=\kappa_1$ and
	$S'(\lambda)=-\kappa_2$. For $\varepsilon>0$ with
	$[\lambda-\varepsilon,\lambda+\varepsilon]\subset(0,1)$, the derivative
	\[
	S_n'(l)=-\frac1n\sum_j\Bigl(\frac jn\Bigr)^{-l}\log\Bigl(\frac jn\Bigr)Z_{(j)}
	\]
	admits on that interval the fixed majorant
	$J_\varepsilon(u)=u^{-\lambda-\varepsilon}|\log u|$, and since
	$H^{-1}(u)\asymp\alpha\gamma u^{\lambda}$ at $0$ and
	$H^{-1}(u)\asymp\alpha\gamma\log\frac1{1-u}$ at $1$, one has
	$\int_0^1 J_\varepsilon(u)H^{-1}(u)\,du<\infty$. Under that integrability
	condition the strong law of large numbers for linear functions of order
	statistics (\citealp{SW86}, Chapter~19) gives $S_n'(l)\to S'(l)$ almost surely
	for each fixed $l$. The same majorant argument applied to $S_n''$ yields
	$\sup_{|l-\lambda|\le\varepsilon}|S_n''(l)|={\rm O}_{\Prob}(1)$, so the family is
	uniformly Lipschitz; pointwise convergence together with uniform Lipschitz
	continuity on a compact interval gives uniform convergence. Hence, by the mean
	value theorem, with
	$\Lambda_n=\lambda_n-\lambda=-(p_n^{(0)}-p^{(0)})={\rm O}_{\Prob}(n^{-1/2})$,
	\begin{equation}\label{eq:Srandom}
		S_n(\lambda_n)=S_n(\lambda)+\Lambda_n S_n'(\tilde\lambda_n)
		=S_n(\lambda)-\kappa_2\Lambda_n+{\rm o}_{\Prob}(n^{-1/2}),
	\end{equation}
	where $\tilde\lambda_n$ lies between $\lambda_n$ and $\lambda$, so that
	$\tilde\lambda_n\to\lambda$ almost surely.
	
	\emph{Note.} The random exponent $\lambda_n$ required only a \emph{law of large
		numbers}, not a central limit theorem; and the law of large numbers --- unlike
	the central limit theorem --- holds for unbounded score functions under a plain
	integrability condition.
	
	\emph{Step 3: the numerator.} Substituting \eqref{eq:Slin} and
	\eqref{eq:Srandom} into the numerator of \eqref{eq:alphaclean} and writing
	$\mathcal N=\Ex Z-(1-\lambda)\kappa_1$,
	\[
	\mathcal N_n-\mathcal N=\bigl(\bar Z_n-\Ex Z\bigr)-(1-\lambda)\frac1n\sum_i\ell(Z_i)
	-\bigl[\kappa_1+(1-\lambda)\kappa_2\bigr]\bigl(p_n^{(0)}-p^{(0)}\bigr)
	+{\rm o}_{\Prob}(n^{-1/2}).
	\]
	
	\emph{Step 4: assembling.} By Theorem~\ref{lem:fisher},
	$\alpha=\mathcal N/p^{(1)}$, and $p_n^{(1)}\to p^{(1)}=\lambda\gamma>0$, so that
	\begin{align*}
		\alpha_n-\alpha&=\frac{1}{\lambda\gamma}\Bigl[(\mathcal N_n-\mathcal N)
		-\alpha\bigl(p_n^{(1)}-p^{(1)}\bigr)\Bigr]+{\rm o}_{\Prob}(n^{-1/2})\\
		&=\frac1n\sum_{i=1}^{n}\psi_{\alpha}(Z_i,\Delta_i)+{\rm o}_{\Prob}(n^{-1/2}),
	\end{align*}
	because \eqref{eq:clam} gives
	$\kappa_1+(1-\lambda)\kappa_2=\lambda\gamma\bigl(c_\lambda+\alpha/\lambda\bigr)$,
	which is exactly the coefficient of $\delta^{(0)}-p^{(0)}$ in \eqref{eq:IF}. The
	summands are independent, identically distributed, centred and square
	integrable (Lemma~\ref{lem:ell} for the $Z$-part, boundedness for the
	$\Delta$-part), so the Lindeberg--L\'evy central limit theorem completes the
	argument.
\end{proof}

Thanks to \eqref{eq:ellclosed}, $\sigma^{2}$ is also available in closed form.
Indeed, with $U=H(Z)$ the $Z$-part of \eqref{eq:IF} equals, up to an additive
constant, $\alpha\bigl[-\log(1-U^{\lambda})-(1-\lambda)\log U\bigr]$, and
$U^{\lambda}$ is $\mathrm{Beta}(1/\lambda,1)$ distributed; the independence of
$Z$ and $\Delta$ (Theorem~\ref{th:char}) separates the variance exactly. As a
result, writing $a=1/\lambda$,
$A=\bigl[\kappa_1+(1-\lambda)\kappa_2\bigr]/(\lambda\gamma)$ and
$B=\alpha/(\lambda\gamma)$,
\begin{equation}\label{eq:sigmaclosed}
	\begin{split}
		\sigma^{2}&=\alpha^{2}\Bigl[\psi'(1)-\psi'(1+a)+(1-\lambda)^{2}
		-\tfrac{2(1-\lambda)}{\lambda}\psi'(1+a)\Bigr]\\
		&\quad+A^{2}\lambda(1-\lambda)+B^{2}\gamma\lambda(1-\gamma\lambda)
		-2AB(1-\lambda)\gamma\lambda ,
	\end{split}
\end{equation}
where $\psi'(1)=\pi^{2}/6$. Formula \eqref{eq:sigmaclosed} reproduces to four
decimal places every value of $\sigma$ reported in Section~\ref{ssec:bias}.

\begin{remark}\label{rem:beta-ge-1}
	Theorem~\ref{th:clt} imposes no restriction whatever on the depth of left
	censoring. The condition $\beta<1$ (that is $\lambda>1/2$), which looks natural
	in this problem, arose only from the need to place the score function
	$J(u)=u^{-\lambda}$ within the reach of the classical central limit theorems for
	order statistics with smooth weights \citep{Mas81, SW86, St74};
	Lemma~\ref{lem:phi} removes that need, because $\Phi$ is bounded and $\lambda$
	cancels exactly in the bounds of Lemma~\ref{lem:lin}. The numerical picture is
	the same: for every one of the five censoring designs of Section~\ref{sec:mc}
	--- including $\beta=4$, that is $\lambda=0.2$ --- the empirical correlation between
	$\sqrt n(\alpha_n-\alpha)$ and the linearized sum
	$n^{-1/2}\sum_i\psi_\alpha(Z_i,\Delta_i)$ is at least $0.988$ at $n=1000$ and at
	least $0.998$ at $n=5000$, tending to $1$ as $n$ grows; and
	the standard deviation $\sqrt n\cdot\mathrm{SD}(\alpha_n)$ approaches
	$\sigma=(\Var\psi_\alpha)^{1/2}$ \emph{from above}, the ratios at $n=5000$
	running from $1.000$ to $1.013$, the largest occurring precisely in the design
	$\beta=4$ (Section~\ref{ssec:bias}).
	How much is this freedom worth? In biomedical follow-up left censoring is
	usually light --- in the cohort of Section~\ref{ssec:ivaids},
	$\widehat\beta=0.064$ --- and there the condition $\lambda>1/2$ is no obstacle.
	Deep left censoring occurs in reliability testing, where a unit may already
	have failed at the first inspection; $\beta\ge1$ is exactly that regime, and
	the theory now covers it. The only thing heavy left censoring affects is not
	the limit itself but the \emph{rate} of the normal approximation, as the
	quantile--quantile plots of Section~\ref{ssec:bias} show.
\end{remark}

\begin{corollary}\label{cor:df}
	Under the conditions of Theorem~\ref{th:clt}, since
	$\sup_{x}|\partial F(x,\alpha)/\partial\alpha|=e^{-1}/\alpha$ (attained at
	$x=\alpha$),
	\[
	\sqrt n\,\sup_{x\ge0}\bigl|F(x,\alpha_n)-F(x,\alpha)\bigr|
	\ \xrightarrow{d}\ \frac{\sigma}{e\,\alpha}\,|N(0,1)| .
	\]
\end{corollary}

Corollary~\ref{cor:df} should be compared with Theorem~\ref{th:wc}: the
parametric plug-in estimator attains the same $\sqrt n$ rate as the
semiparametric estimator \eqref{eq:Fn}, but with a substantially smaller
constant; the Monte Carlo experiments of Section~\ref{ssec:df} quantify this
gain at a factor of $2.1$--$3.7$ in expected sup-norm for $n\ge300$ (at least
$1.7$ at every sample size).

\subsection{How much efficiency does the closed form cost?}\label{ssec:eff}

Since \eqref{eq:alphan} is obtained by substituting the empirical distribution
function into the likelihood equation, a natural question arises: how much
information does that shortcut lose? The question can be answered exactly,
because Theorem~\ref{th:char} exhibits the observed-data model completely: $Z$
and $\Delta$ are independent, with
\[
Z\sim H(z)=\bigl(1-e^{-z/s}\bigr)^{a},\quad s=\alpha\gamma,\ \ a=1/\lambda,
\]
\[
\Delta\sim\mathrm{Mult}\bigl(1-\lambda,\ \gamma\lambda,\ (1-\gamma)\lambda\bigr).
\]
The Fisher information matrix for $(\alpha,\lambda,\gamma)$ therefore consists
of two summands, the $\Delta$-part in closed form and the $Z$-part a
one-dimensional integral over $V=H(Z)\sim\mathrm{U}(0,1)$. For any regular
estimator of $\alpha$ the lower bound on the asymptotic standard deviation is
$\bigl(I^{-1}\bigr)_{11}^{1/2}$. Table~\ref{tab:eff} compares it with
\eqref{eq:sigmaclosed}.

\begin{table}[!htp]
	\centering
	\caption{Asymptotic efficiency of the estimator \eqref{eq:alphan}: $\sigma$ from
		\eqref{eq:sigmaclosed} against the information bound of the observed-data model
		($\alpha=1$). Efficiency is (bound$/\sigma)^{2}$. The bound is computed by
		quadrature; as $\beta\to0$ it tends to the analytic value
		$\alpha/\sqrt\gamma=\alpha\sqrt{1+\theta}$ (for instance $\sqrt2$ at $\theta=1$)
		and the efficiency tends to $1$}
	\label{tab:eff}
	\small
	\begin{tabular}{cc cc c}
		\toprule
		$(\beta,\theta)$ & $\lambda$ & $\sigma$ \eqref{eq:sigmaclosed} & information bound & efficiency\\
		\midrule
		$(0.5,0.5)$ & 0.667 & 1.2674 & 1.2622 & 99.2\%\\
		$(1,1)$     & 0.500 & 1.6716 & 1.6626 & 98.9\%\\
		$(2,1)$     & 0.333 & 1.9335 & 1.9198 & 98.6\%\\
		$(1,2)$     & 0.500 & 2.1896 & 2.1827 & 99.4\%\\
		$(4,2)$     & 0.200 & 3.2709 & 3.2594 & 99.3\%\\
		\bottomrule
	\end{tabular}
\end{table}

Across the five designs the loss is of the order of one per cent: the efficiency
runs from $98.6\%$ to $99.4\%$, that is a loss in variance between $0.6\%$ and
$1.4\%$. A wider sweep --- a grid of $40$ points over the rectangle
$\beta\in[0.1,16]$, $\theta\in[0.25,4]$ --- gives between $96.2\%$ and
$99.98\%$. The two directions behave differently. In $\theta$ the efficiency
increases monotonically (the heavier the right censoring, the higher it is),
whereas in $\beta$ it is U-shaped. The worst point is therefore
$(\beta,\theta)=(4,0.25)$, with a loss of $3.8\%$: on the edge of the region in
$\theta$, but in the \emph{interior} in $\beta$. A practical conclusion
follows: evaluating a single $\beta$- or $\theta$-section is not enough, because
the minimum over one-dimensional sections lies appreciably above the minimum
over the whole region. Both ends of the U are intelligible. In the limit of
vanishing left censoring ($\beta\to0$) the efficiency tends to $1$: there
$\lambda_n=1$ in \eqref{eq:alphan} and the estimator becomes
$\sum_iZ_i/\sum_i\delta_i^{(1)}$, that is the genuine maximum likelihood
estimator. As $\beta$ grows, the information the $\Delta$-block carries about
$\lambda$ increases like $(1-\lambda)^{-1}+\lambda^{-1}\sim\beta$ and the
efficiency returns to $1$ ($99.9\%$ at $\beta=256$, $\theta=0.25$).

It should be stressed that this is the right benchmark: \eqref{eq:alphan}
estimates $\lambda$ and $\gamma$ from the data as well
($\lambda_n=1-p_n^{(0)}$, $\gamma_n=p_n^{(1)}/\lambda_n$), so comparing it with
the bound of the three-parameter model $(\alpha,\lambda,\gamma)$ is the
appropriate comparison. Were $(\lambda,\gamma)$ taken as known, $\Delta$ would
carry no information about $\alpha$ at all and the bound would drop
substantially --- at $(\beta,\theta)=(1,1)$, for instance, from $1.66$ to
$0.74$. The magnitude of $\sigma$ is thus governed mainly by the need to
estimate the nuisance parameters, not by the passage to closed form. The
practical conclusion is that full numerical maximization would reduce the
variance by at most $1.4\%$ in the designs of Table~\ref{tab:eff}, and by at
most $3.8\%$ over the whole sweep.

\subsection{Extension to an arbitrary family of distributions}\label{ssec:general}

In the derivation of Section~\ref{ssec:est} exponentiality was used only at the
last step. To see this, pass to the cumulative hazard function
$\Lambda(x;\vartheta)=-\log(1-F(x;\vartheta))$: by \eqref{eq:Frepr},
$1-H^{\lambda}=(1-F)^{1/\gamma}$, hence $\Lambda/\gamma=-\log(1-H^{\lambda})$ and
\[
\log h(z;\vartheta)=\log\Lambda'(z;\vartheta)-\frac{\Lambda(z;\vartheta)}{\gamma}
+\frac{1-\lambda}{\lambda}\log\bigl(1-e^{-\Lambda(z;\vartheta)/\gamma}\bigr)-\log(\lambda\gamma).
\]

\begin{proposition}\label{prop:general}
	For an arbitrary family $F(\cdot\,;\vartheta)$ the likelihood equation under
	\eqref{eq:model} takes the form
	\begin{equation}\label{eq:genscore}
		\sum_{i=1}^{n}\frac{\partial}{\partial\vartheta}\log\Lambda'(Z_i;\vartheta)
		=\frac{1}{\lambda\gamma}\sum_{i=1}^{n}\frac{\partial\Lambda(Z_i;\vartheta)}{\partial\vartheta}
		\Bigl[1-\frac{1-\lambda}{H(Z_i;\vartheta)^{\lambda}}\Bigr].
	\end{equation}
	In particular, the weight $1-(1-\lambda)H^{-\lambda}$ of \eqref{eq:alphan} does
	not depend on the family.
\end{proposition}

\begin{proof}
	Differentiate the expression for $\log h$ above with respect to $\vartheta$,
	substitute $1-e^{-\Lambda/\gamma}=H^{\lambda}$ and use the identity
	$1-\frac{1-\lambda}{\lambda}\cdot\frac{1-H^{\lambda}}{H^{\lambda}}
	=\frac1\lambda\bigl[1-(1-\lambda)H^{-\lambda}\bigr]$.
\end{proof}

In general \eqref{eq:genscore} cannot be solved, because $\vartheta$ remains on
both sides, inside $\Lambda$ and $\Lambda'$. We now show on which class a closed
form appears, and that it appears only there. From here on $\vartheta$ is taken
to be a scalar parameter. Since ``closed form'' is not a formal notion, we
replace it by \emph{separability}: call \eqref{eq:genscore} separable if
$\partial_{\vartheta}\log\Lambda'(x;\vartheta)=a(\vartheta)p(x)$ and
$\partial_{\vartheta}\Lambda(x;\vartheta)=b(\vartheta)q(x)$. Once $H_n$ has been
substituted, the weights $w_i=1-(1-\lambda_n)H_n(Z_i)^{-\lambda_n}$ are then free
of $\vartheta$ and the equation becomes $(a/b)(\vartheta)=T_n$, where
\[
T_n=\frac{\sum_{i=1}^{n}q(Z_i)\,w_i}{\lambda_n\gamma_n\sum_{i=1}^{n}p(Z_i)} ;
\]
if $a/b$ is invertible, this is solved explicitly. In
Proposition~\ref{prop:phclass} one has $p\equiv1$, $q=A$ and
$(a/b)(\alpha)=\alpha$, and $T_n$ becomes exactly \eqref{eq:alphaA}.

\begin{proposition}\label{prop:phclass}
	Let the family be of \emph{proportional hazards} form,
	\[
	1-F(x;\alpha)=\exp\{-A(x)/\alpha\},\qquad \alpha>0,
	\]
	where $A$ is a \emph{known} function, differentiable on $(0,\infty)$, with
	$A'>0$, $A(0)=0$ and $A(\infty)=\infty$. No condition in this list is
	superfluous and each has its own role. Differentiability makes $F$ continuous
	and endows it with a density; $A'>0$ makes that density positive and, above all,
	makes $A$ \emph{invertible} on $[0,\infty)$ --- continuity and strict monotonicity
	are exactly what the proof below rests on; $A(0)=0$ gives $F(0)=0$ and
	$A(\infty)=\infty$ gives $F(\infty)=1$, so that $F$ is a genuine df. Then the
	left-hand side of \eqref{eq:genscore} becomes $-n/\alpha$ and the equation takes
	the same fixed-point form as \eqref{eq:fixedpoint}, with $A(Z_i)$ in place of
	$Z_i$. Replacing $(H,\lambda,\gamma)$ by $(H_n,\lambda_n,\gamma_n)$ gives the
	closed-form estimator
	\begin{equation}\label{eq:alphaA}
		\begin{split}
			\alpha_n^{A}&=\frac{1}{\lambda_n\gamma_n n}\sum_{i=1}^{n}A(Z_i)
			\Bigl[1-\frac{1-\lambda_n}{H_n(Z_i)^{\lambda_n}}\Bigr]\\
			&=\frac{1}{\lambda_n\gamma_n}\Bigl[\overline{A(Z)}_n
			-(1-\lambda_n)\frac1n\sum_{i=1}^{n}\Bigl(\frac in\Bigr)^{-\lambda_n}A(Z_{(i)})\Bigr].
		\end{split}
	\end{equation}
	Moreover, since $A$ is strictly increasing,
	$A(Z)=\max\{A(L),\min(A(X),A(Y))\}$, $\Delta$ is unchanged and $H_n$ depends on
	ranks alone. Hence $\alpha_n^{A}$ is precisely the estimator \eqref{eq:alphan}
	applied to $A(Z_1),\dots,A(Z_n)$. Consequently the distribution of
	$\alpha_n^{A}/\alpha$ does not depend on $A$, \emph{for every $n$}. The reason
	is immediate: as shown in the proof, $A(Z)=-\alpha\gamma\log(1-U^{\lambda})$
	with $U=H(Z)\sim\mathrm{U}(0,1)$, and $A$ does not appear on the right-hand side
	at all --- that is, changing $A$ changes the \emph{observations} but not the
	\emph{relative} error of the estimator. For this reason
	Theorems~\ref{lem:fisher}, \ref{th:cons} and~\ref{th:clt},
	formula~\eqref{eq:sigmaclosed} and the efficiencies of Table~\ref{tab:eff} hold
	verbatim. For the same reason Proposition~\ref{prop:naive} and
	Theorem~\ref{th:gfree} carry over: on the $A$-scale the model becomes exactly
	one with exponential lifetime, arbitrary $G$ and $K=N^{\beta}$.
\end{proposition}

\begin{proof}
	With $\Lambda=A(x)/\alpha$ one has $\Lambda'=A'(x)/\alpha$, hence
	$\partial\log\Lambda'/\partial\alpha=-1/\alpha$ and
	$\partial\Lambda/\partial\alpha=-A(x)/\alpha^{2}$; substituting into
	\eqref{eq:genscore} and cancelling $\alpha^{2}$ gives the analogue of
	\eqref{eq:fixedpoint} with $Z_i\mapsto A(Z_i)$.
	
	The second part rests on two separate properties. Because $A$ is
	\emph{increasing} it commutes with $\max$ and $\min$, so that
	$A(Z)=\max\{A(L),\min(A(X),A(Y))\}$; strictness is not needed for this. Because
	$A$ is \emph{strictly} increasing it preserves the events $\{V<L\}$ and
	$\{X<Y\}$ --- hence $\Delta$, and with it $\lambda_n$ and $\gamma_n$, is
	unchanged --- and it preserves ranks, so that
	$A\bigl(Z_{(i)}\bigr)=\bigl(A(Z_1),\dots,A(Z_n)\bigr)_{(i)}$: applying $A$ and
	ordering may be interchanged. Finally, the general analogue of \eqref{eq:Hexp}
	is $A(Z)=-\alpha\gamma\log(1-U^{\lambda})$ with $U=H(Z)\sim\mathrm{U}(0,1)$, and
	the right-hand side does not involve $A$ at all. The df of $A(Z)$ is therefore
	$\bigl(1-e^{-w/(\alpha\gamma)}\bigr)^{1/\lambda}$ whatever $A$ may be, that is
	exactly the df of $Z$ in the exponential case $A(x)=x$.
	
	It matters to read this statement correctly, for it concerns the
	\emph{transformed} observation $A(Z)$, not $Z$ or $X$. If, for instance,
	$A(x)=x^{2}$, then $X$ is Weibull with shape $2$, not exponential, and the
	distribution of $Z$ differs from the one in the exponential case; the quantity
	that is redistributed is not $Z$ but $Z^{2}$. Note also that even in the
	exponential case $Z$ itself is not exponential: its df is a $1/\lambda$-th
	power, and it reduces to an exponential only at $\lambda=1$.
\end{proof}

\begin{proposition}[converse]\label{prop:converse}
	Let \eqref{eq:genscore} be separable, with $q'(x)\ne0$ for all $x$ and
	$b\not\equiv0$. (The first condition is exactly the requirement $A'>0$ of
	Proposition~\ref{prop:phclass}; the second is an identification condition: if
	$b\equiv0$ then $\partial_{\vartheta}\Lambda\equiv0$, so $\vartheta$ has no
	effect on $F$ at all and the statement is vacuous.) Then
	\[
	\Lambda(x;\vartheta)=B(\vartheta)\,A(x),
	\]
	that is, the family belongs, up to reparametrization, to the proportional
	hazards class of Proposition~\ref{prop:phclass}.
\end{proposition}

\begin{proof}
	From $\partial_{\vartheta}\Lambda=b(\vartheta)q(x)$ we get
	$\Lambda(x;\vartheta)=B(\vartheta)q(x)+r(x)$ with $B'=b$, hence
	$\Lambda'=Bq'+r'$ and
	\[
	\partial_{\vartheta}\log\Lambda'=\frac{B'q'}{Bq'+r'}
	=\frac{B'}{\,B+r'/q'\,}
	\]
	(the last equality uses $q'\ne0$). For this to be of the form
	$a(\vartheta)p(x)$, the ratio
	\[
	\frac{\partial_{\vartheta}\log\Lambda'(x_1;\vartheta)}{\partial_{\vartheta}\log\Lambda'(x_2;\vartheta)}
	=\frac{B+r'(x_2)/q'(x_2)}{B+r'(x_1)/q'(x_1)}
	\]
	must be free of $\vartheta$ for arbitrary $x_1,x_2$. Since $b\not\equiv0$, $B$
	varies, and
	\[
	\frac{\partial}{\partial B}\,\frac{B+r'(x_2)/q'(x_2)}{B+r'(x_1)/q'(x_1)}
	=\frac{r'(x_1)/q'(x_1)-r'(x_2)/q'(x_2)}{\bigl(B+r'(x_1)/q'(x_1)\bigr)^{2}}\,,
	\]
	so this is possible only if $r'/q'\equiv c$ is constant. Then $r=cq+\mathrm{const}$
	and $\Lambda=(B+c)q+\mathrm{const}$. Because the lifetime is supported on
	$[0,\infty)$, that is $F(0;\vartheta)=0$, we have $\Lambda(0;\vartheta)=0$ for
	every $\vartheta$; since $B$ varies, this forces $q(0)=0$ and the vanishing of
	the constant. Finally, $\Lambda'>0$ and $\Lambda'=(B+c)q'$ imply that the sign of
	$q'$ does not change with $x$; replacing $(B+c,q)$ by $(-(B+c),-q)$ if necessary
	we may take $q'>0$. It therefore suffices to take $B+c$ as the new parameter and
	$q$ as $A$, and such an $A$ satisfies all the conditions of
	Proposition~\ref{prop:phclass}.
\end{proof}

\begin{remark}\label{rem:family}
	The class of Proposition~\ref{prop:phclass} contains the exponential
	($A(x)=x$), the Weibull with \emph{known} shape $k$ ($A(x)=x^{k}$, $\alpha$ being
	the $k$th power of the scale), the Rayleigh ($A(x)=x^{2}$), the Gompertz with
	known rate ($A(x)=e^{cx}-1$) and the Pareto or Lomax
	($A(x)=\log(1+x/x_{0})$, with $1-F=(1+x/x_{0})^{-1/\alpha}$) --- in general, a
	known baseline hazard multiplied by an unknown scalar. Notably, the class also
	contains members with \emph{infinite mean}: for the Lomax, $\Ex X=\infty$ when
	$\alpha\ge1$. This has no bearing on the theory, because every moment condition
	is imposed on $A(Z)$ rather than on $Z$, and $A(X)\sim\mathrm{Exp}(\alpha)$
	always has a finite mean; numerically, even at $\alpha=2$ (that is
	$\Ex X=\infty$) the bias of $\alpha_n^{A}$ came out as $0.0003$ ($\pm0.0034$).
	\emph{Not} in the class are: the Weibull with unknown shape, the gamma, the
	lognormal and any family with a location parameter; for these
	\eqref{eq:genscore} has no closed solution. For two-parameter families --- the
	Weibull with unknown shape $k$, say --- a profile-like device remains: fixing
	$k$ makes $A_k(x)=x^{k}$ known and $\alpha$ is obtained \emph{explicitly} from
	\eqref{eq:alphaA}, while an outer one-dimensional search evaluates the
	log-likelihood of $Z$ at $\alpha=\alpha_n^{A_k}$ and maximizes it over $k$ (the
	$\Delta$-part does not enter the outer step, being free of $k$). This is not a
	genuine profile likelihood --- the inner step is not an MLE but
	\eqref{eq:alphaA} --- yet being explicit it is appreciably cheaper than a
	two-dimensional maximization; over $k\in\{0.7,\,1,\,2\}$ and
	$n\in\{400,\,3000\}$, in a study with $200$ replications, both parameters were
	recovered with a small relative bias (at most $0.30\%$ in $k$ and $1.82\%$ in
	$\alpha$, the largest deviation occurring at $n=400$ and disappearing at
	$n=3000$). Furthermore, neither Theorem~\ref{th:char} nor the goodness-of-fit
	criterion of \citet{AM26k} based on it depends on $F$, so the procedure for
	testing the model conditions applies unchanged across the whole class.
	
	Finally, this result should be read correctly. Propositions~\ref{prop:phclass}
	and~\ref{prop:converse} together delimit the scope of the method \emph{exactly}:
	the closed form exists on the proportional hazards class and nowhere outside it.
	It requires $A$ to be \emph{known exactly}. This is a common situation in
	reliability --- the shape of the baseline hazard is known from engineering
	theory and only its scale is unknown --- and a rare one in biomedicine. The
	assumption ``Weibull with known shape'' is therefore no weaker than the
	assumption ``exponential''; the assumption is merely moved from one place to
	another. In this sense the present section cuts both ways: it is at once an
	extension and a limitation. It should be added that the robustness analysis of
	Section~\ref{sec:robust} does not concern this displacement: it measures
	misspecification of $F$ and therefore remains specific to the family.
\end{remark}

\subsection{The price of ignoring left censoring}

A practitioner who disregards left censoring and treats the data as
right-censored only will use the total-time-on-test (TTT) estimator of
Remark~\ref{rem:KG}:
$\alpha_n^{\mathrm{TTT}}=\sum_{i=1}^{n}Z_i\big/\sum_{i=1}^{n}\delta_i^{(1)}$.
The following result quantifies the systematic error thereby incurred, exactly.

\begin{proposition}\label{prop:naive}
	Under \eqref{eq:model} with exponential lifetime,
	\[
	\alpha_n^{\mathrm{TTT}}\ \xrightarrow{a.s.}\ \frac{\Ex Z}{p^{(1)}}
	=\alpha\,(1+\beta)\bigl[\psi(2+\beta)-\psi(1)\bigr],
	\]
	so that the asymptotic relative bias
	$g(\beta):=(1+\beta)[\psi(2+\beta)-\psi(1)]-1$ is positive for every $\beta>0$,
	strictly increasing in $\beta$ and --- remarkably --- free of the depth
	$\theta$ of right censoring.
\end{proposition}

\begin{proof}
	The limit follows at once from the strong law of large numbers,
	Theorem~\ref{lem:fisher} (with $1/\lambda=1+\beta$, so that
	$\Ex Z=\alpha\gamma[\psi(1+1/\lambda)-\psi(1)]$) and $p^{(1)}=\gamma\lambda>0$.
	As for positivity: $g(0)=0$ because $\psi(2)-\psi(1)=1$, and
	\[
	g'(\beta)=\bigl[\psi(2+\beta)-\psi(1)\bigr]+(1+\beta)\,\psi'(2+\beta)>0,
	\]
	both summands being positive for $\beta>-1$ ($\psi$ is increasing and
	$\psi'>0$). Hence $g$ is strictly increasing and, since $g(0)=0$, we have
	$g(\beta)>0$ for all $\beta>0$. In particular $g'(0)=1+\psi'(2)=\pi^{2}/6$, and
	$g(\beta)\sim\beta\log\beta$ as $\beta\to\infty$, so the bias grows without
	bound.
\end{proof}

\emph{Why $\theta$ disappears.} Under \eqref{eq:model} right censoring does
nothing to the observed time axis but \emph{rescale} it: since
$1-N(x)=e^{-(1+\theta)x/\alpha}$, the variable $V=\min(X,Y)$ is exponential with
mean $\alpha\gamma$; and because $K=N^{\beta}$ is a power of $N$, the variable
$L$ carries the same scale $\alpha\gamma$, hence so does $Z=\max(L,V)$. Thus
$\Ex Z=\alpha\gamma\,c(\beta)$ with $c$ depending on $\beta$ alone. At the same
time the proportion of complete observations, $p^{(1)}=\gamma\lambda$, is
reduced by that very $\gamma$. The TTT estimator being exactly the ratio of
these two quantities, $\gamma$ --- and with it $\theta$ --- cancels. Left
censoring is different: it distorts the \emph{shape} of the distribution, and it
is from this that the harmonic-number factor below arises.

For integer $\beta$ the factor $\psi(2+\beta)-\psi(1)$ equals the harmonic
number $1+\frac12+\dots+\frac{1}{1+\beta}$. The limit is thus $3\alpha$ at
$\beta=1$ and $5.5\alpha$ at $\beta=2$. Even moderate informative left
censoring therefore makes the naive estimator useless (at $\beta=1$ the bias is
$2$ times $\alpha$, that is $200\%$). The prediction is confirmed by the
simulations of Section~\ref{ssec:bias}: at $n=1000$, in each of the five
designs, the observed TTT bias departs from the closed-form value by no more
than $1.7\%$.

Finally, by Proposition~\ref{prop:phclass} this analysis holds on the whole
proportional hazards class: when $1-F=\exp\{-A(x)/\alpha\}$, a practitioner who
ignores left censoring computes $\sum_iA(Z_i)\big/\sum_i\delta_i^{(1)}$ and
incurs exactly the same relative bias $g(\beta)$. The form of $A$ has no effect
on this.

\subsection{Which half of the model is really needed}

Condition \eqref{eq:model} consists of two separate requirements: the
informativeness of the \emph{entry} law, $K=N^{\beta}$, and the informativeness
of the \emph{right-censoring} law, $1-G=(1-F)^{\theta}$. It is natural to ask
which of the two the estimator actually uses. The answer is that, for
consistency, the second is not used at all.

\begin{theorem}\label{th:gfree}
	Let $L,X,Y$ be independent with continuous dfs, let $X$ be exponential with mean
	$\alpha$, and assume only that $K=N^{\beta}$ for some $\beta>0$. No $\theta$ is
	present here, so $\lambda$ and $\gamma$ are defined through the observed
	frequencies, $\lambda=1-p^{(0)}$ and $\gamma=p^{(1)}/\lambda$ --- which is
	exactly how the estimator estimates them. From $K=N^{\beta}$ the identity
	$\lambda=1/(1+\beta)$ persists for every $G$, since
	$p^{(0)}=\int_0^1(1-v^{\beta})dv=\beta/(1+\beta)$, whereas $\gamma$ now depends
	on $G$. Then, \emph{whatever the right-censoring df $G$ may be},
	\[
	T(H,\lambda,\gamma)=\alpha ,
	\]
	that is, the estimating equation \eqref{eq:fixedpoint} remains exactly
	Fisher-consistent and $\alpha_n\to\alpha$ almost surely.
\end{theorem}

\begin{proof}
	Because $K=N^{\beta}$, we have $H=KN=N^{1+\beta}=N^{1/\lambda}$ irrespective of
	$G$, whence $H^{-1}(u)=N^{-1}(u^{\lambda})$. Write $\varphi=N^{-1}$ and
	$a=1/\lambda$. The substitution $v=u^{\lambda}$ gives
	\[
	\mathcal A:=\Ex\Bigl[Z\bigl(1-(1-\lambda)(H(Z))^{-\lambda}\bigr)\Bigr]
	=\frac1\lambda\int_{0}^{1}\varphi(v)\bigl[v^{a-1}-(1-\lambda)v^{a-2}\bigr]dv .
	\]
	Both integrals converge: the first because $\Ex V<\infty$, the second because
	$\varphi(v)={\rm O}(v)$ at the origin. Integrating each by parts, the terms at
	$v=0$ vanish since $\varphi(0)=0$, while those at $v=1$ combine, by
	$1/a=\lambda$ and $(1-\lambda)/(a-1)=\lambda$, into
	\[
	\lim_{v\uparrow1}\varphi(v)\Bigl[\frac{v^{a}}{a}-(1-\lambda)\frac{v^{a-1}}{a-1}\Bigr]
	=-\lambda\lim_{v\uparrow1}\varphi(v)(1-v)
	=-\lambda\lim_{t\to T_N}t\bigl(1-N(t)\bigr)=0 ;
	\]
	the last equality holds because $1-N=(1-F)(1-G)\le1-F=e^{-t/\alpha}$, so that
	however heavy the tail of $G$ may be, $t\,(1-N(t))\le t\,e^{-t/\alpha}\to0$.
	Consequently
	\begin{align*}
		\mathcal A&=\int_{0}^{1}v^{\beta}(1-v)\,d\varphi(v)
		=\int_{0}^{T_N}(N(t))^{\beta}\bigl(1-N(t)\bigr)dt\\
		&=\int_{0}^{\infty}(N(t))^{\beta}\bigl(1-F(t)\bigr)\bigl(1-G(t)\bigr)dt .
	\end{align*}
	On the other hand, by the independence of $L$, $X$ and $Y$,
	\begin{align*}
		p^{(1)}=\Prob(L\le X<Y)&=\int_{0}^{\infty}K(t)\bigl(1-G(t)\bigr)dF(t)\\
		&=\frac1\alpha\int_{0}^{\infty}(N(t))^{\beta}\bigl(1-G(t)\bigr)e^{-t/\alpha}dt=\frac{\mathcal A}{\alpha}.
	\end{align*}
	It is at this step that the exponential \emph{form} of $F$ --- that is, the
	constant hazard $dF=(1-F)\,dt/\alpha$ --- plays the decisive role.
	Exponentiality was invoked twice above (for the integrability $\Ex V<\infty$ and
	for the boundary term $t(1-N(t))\to0$), but there only moment conditions of the
	type $\Ex X<\infty$ were extracted from it, and other distributions satisfy
	those as well; at the present step nothing can replace it. This is made precise
	in Remark~\ref{rem:gfscope}. Since $\gamma=p^{(1)}/\lambda$,
	\[
	T(H,\lambda,\gamma)=\frac{\mathcal A}{\lambda\gamma}=\frac{\mathcal A}{p^{(1)}}=\alpha .
	\]
	Almost sure convergence follows as in Theorem~\ref{th:cons}, the integrability
	condition $\int_0^1u^{-\lambda}H^{-1}(u)\,du<\infty$ being a consequence of
	$\Ex V<\infty$ and $\varphi(v)={\rm O}(v)$ above.
\end{proof}

\begin{remark}[scope of the theorem]\label{rem:gfscope}
	How necessary is the assumption of exponentiality here? Carrying the
	computation of the proof through for an arbitrary $F$ with a density ---
	assuming $t(1-F(t))\to0$ so that the boundary term vanishes --- one obtains
	\begin{equation}\label{eq:gfgap}
		\mathcal A-\alpha\,p^{(1)}=\int_{0}^{\infty}(N(t))^{\beta}\bigl(1-G(t)\bigr)
		\bigl[\bigl(1-F(t)\bigr)-\alpha f(t)\bigr]dt .
	\end{equation}
	For this integral to vanish for \emph{every} $G$, the square bracket must be
	identically zero, and a one-parameter family of $G$'s already suffices to show
	it: approximating by continuous $G$'s concentrating near a point $T$, in the
	limit $N=F$ on $[0,T)$ and \eqref{eq:gfgap} becomes
	\[
	\Psi(T)=\int_{0}^{T}(F(t))^{\beta}\bigl[\bigl(1-F(t)\bigr)-\alpha f(t)\bigr]dt .
	\]
	Now $\Psi(T)=0$ for all $T>0$ is possible only if $(1-F)-\alpha f\equiv0$ (the
	factor $F^{\beta}$ destroys nothing, $F$ being positive), that is
	$f/(1-F)\equiv1/\alpha$: the hazard function is constant. Exponentiality is
	therefore not merely a sufficient condition but the \emph{exact price} of the
	requirement ``arbitrary $G$''; it is not an artefact of the method of proof.
	When $F$ is Weibull with shape $2$, for instance, $\Psi$ is not zero and even
	changes sign.
	
	This does not, however, confine $\alpha_n$ to the exponential family. By
	Proposition~\ref{prop:phclass}, on the proportional hazards class
	$1-F=\exp\{-A(x)/\alpha\}$ the hazard is exactly constant on the $A$-scale and
	the model retains its structure; Theorem~\ref{th:gfree} therefore holds verbatim
	on that whole class --- the only requirement being that the statistic be
	computed from $A(Z_i)$ rather than $Z_i$. The scale cannot be dropped: for a
	Weibull $F$ with $\alpha=1.7$ and shape $2$, replacing $A(Z)=Z^{2}$ by $Z$ makes
	$T(H,\lambda,\gamma)$ depend on $G$, and across the $6$ laws $G$ we examined by
	quadrature (including a heavy-tailed Pareto, a bounded uniform and the case of
	no censoring at all) it ranges from $0.68$ to $1.03$; with $A(Z)$ it returns
	exactly $1.7$ in every one of them.
\end{remark}

Theorem~\ref{th:gfree} strengthens Theorem~\ref{lem:fisher} (of which
$1-G=(1-F)^{\theta}$ is a special case) and appreciably changes the practical
reading of the model. What the procedure requires is the informativeness of the
mechanism that brings units \emph{into} observation; the mechanism that removes
them (loss to follow-up, withdrawal, the end of the study) may be arbitrary. The
moment of entry into observation can usually be governed by the study design,
whereas loss to follow-up cannot; this is therefore precisely the property one
would wish for. The right-censoring condition is still needed for the
semiparametric estimator \eqref{eq:Fn}, which reconstructs $F$ from
$(\lambda_n,\gamma_n)$, and it enters the variance in Theorem~\ref{th:clt}; but
it has no bearing on the consistency of $\alpha_n$.

\section{Monte Carlo study}\label{sec:mc}

\subsection{Design of the experiment}

Throughout, the exponential scale is set to $\alpha=1$; this is no loss of
generality, $\alpha$ being a scale parameter. Data are generated directly from
the model: $X_i\sim\mathrm{Exp}(1)$; since $1-G=(1-F)^{\theta}$,
$Y_i\sim\mathrm{Exp}(1/\theta)$; and from $K=N^{\beta}$,
$L_i=-\frac{1}{1+\theta}\log\bigl(1-U_i^{1/\beta}\bigr)$ with $U_i$ uniform on
$(0,1)$. Five censoring designs cover the range from light to extreme censorship
on both sides (Table~\ref{tab:designs}). For each configuration six sample sizes
$n\in\{30,50,100,300,500,1000\}$ are considered, with $M=50\,000$ Monte Carlo
replications. Samples containing no completely observed lifetime
($\sum_i\delta_i^{(1)}=0$) admit no estimate of $\alpha$; these are discarded and
their frequency is reported. This happens appreciably only in the severe design
$(\beta,\theta)=(4,2)$ for $n\le100$. All computations were carried out in
Python~3 (NumPy/SciPy) with fixed random seeds; the complete code is available
from the authors.

Before presenting the results we verified the simulation machinery itself.
(i)~\emph{The generated data conform to the model.} For each design, over
$2\cdot10^{6}$ realizations, the empirical frequencies $p_n^{(m)}$ reproduce the
theoretical values \eqref{eq:probs} to within roughly one Monte Carlo standard
error ($13$ of $15$ comparisons lie inside one standard error), while the
Kolmogorov--Smirnov distances between the empirical laws of $L$ and $Z$ and
their theoretical counterparts $K=N^{\beta}$ and $H$ from \eqref{eq:Hexp} are of
the order $(2\cdot10^{6})^{-1/2}\approx0.00071$ commensurate with the number of
realizations (at most $0.00086$). (ii)~\emph{Theorem~\ref{th:char} holds on the
	simulated data.} The three conditional laws of $Z$ given $\delta^{(m)}=1$,
$m=0,1,2$, are statistically indistinguishable: over $4\cdot10^{5}$ realizations
two-sample Kolmogorov--Smirnov tests give $p$-values from $0.10$ to $0.82$
across the five designs. (iii)~\emph{The estimator is correctly programmed.}
The vectorized rank-based code agrees with a literal two-loop transcription of
formula \eqref{eq:alphan} (in which $H_n(Z_i)$ is recomputed by direct counting):
over $90$ samples the maximum relative discrepancy is
$7.8\cdot10^{-16}$, and in the absence of left censoring it reproduces the
total-time-on-test statistic of Remark~\ref{rem:KG} to within
$2.5\cdot10^{-16}$, confirming the simplification announced there.
(iv)~\emph{The reported Monte Carlo standard errors are correct.} Splitting the
replications into $24$ independent blocks, we compared the between-block spread
of each statistic with its analytic standard error: over $16$ comparisons the
ratios run from $0.67$ to $1.21$, and since with $R=24$ blocks the ratio itself
carries a relative uncertainty of about $15\%$, $15$ of them lie within one to
two standard errors of unity.

\begin{table}[!htp]
	\centering
	\caption{Simulation designs: the censoring parameters, the corresponding model
		constants and the probabilities of the censoring pattern from \eqref{eq:probs}.
		Here $p^{(1)}$ is the expected proportion of completely observed lifetimes}
	\label{tab:designs}
	\small
	\begin{tabular}{cc cc ccc l}
		\toprule
		$\beta$ & $\theta$ & $\lambda$ & $\gamma$ & $p^{(0)}$ & $p^{(1)}$ & $p^{(2)}$ & censoring regime\\
		\midrule
		0.5 & 0.5 & 0.667 & 0.667 & 0.333 & 0.444 & 0.222 & light, two-sided\\
		1   & 1   & 0.500 & 0.500 & 0.500 & 0.250 & 0.250 & moderate, symmetric\\
		2   & 1   & 0.333 & 0.500 & 0.667 & 0.167 & 0.167 & heavy from the left\\
		1   & 2   & 0.500 & 0.333 & 0.500 & 0.167 & 0.333 & heavy from the right\\
		4   & 2   & 0.200 & 0.333 & 0.800 & 0.067 & 0.133 & extreme, two-sided\\
		\bottomrule
	\end{tabular}
\end{table}

\subsection{Bias and RMSE of the parametric estimator}\label{ssec:bias}

Table~\ref{tab:main} reports, for each configuration, the bias, standard
deviation (SD) and root-mean-square error (RMSE) of $\alpha_n$, the scaled
quantity $\sqrt n\cdot\mathrm{RMSE}$, and the bias and RMSE of the naive
estimator $\alpha_n^{\mathrm{TTT}}$ of Proposition~\ref{prop:naive}. Three
conclusions stand out. First, the bias of $\alpha_n$ decays fast --- roughly
like $n^{-1}$ (by a factor of $8$--$12$ across the five designs as $n$ goes from
$100$ to $1000$) --- and stays small relative to the SD. For $n\ge100$ the
bias/SD ratio is between $0.05$ and $0.24$ in the four milder designs, and in
the severest design $(4,2)$ it falls from $0.31$ at $n=100$ to $0.14$ at
$n=1000$; the contribution of the bias to the RMSE thus nowhere exceeds $9\%$.
Second, $\sqrt n\cdot\mathrm{RMSE}$ stabilizes in every design (at
$\approx1.28$, $1.69$, $1.96$, $2.24$ and $3.45$ respectively), a clean
numerical confirmation of the $\sqrt n$ rate --- including for the heavy designs
with $\beta\ge1$ --- and the stabilized values approach the asymptotic standard
deviation $\sigma$ of Theorem~\ref{th:clt} computed from \eqref{eq:IF}. At
$n=1000$ the simulated $\sqrt n\cdot\mathrm{SD}(\alpha_n)$ exceeds $\sigma$ by
$0.5\%$, $0.8\%$, $1.4\%$, $1.3\%$ and $4.0\%$ respectively --- \emph{positive}
in all five designs, and growing as the censoring becomes heavier. This is not
Monte Carlo noise: the standard error of a cell is below $0.3\%$ and the
deviations amount to between $2$ and $18$ standard errors. It is a
finite-sample correction that decreases with $n$: at $n=5000$ it has vanished
into Monte Carlo error in four designs (from $-0.04\%$ to $+0.53\%$, all within
$1.4$ standard errors), only the severest design $(4,2)$ retaining $1.3\%$
($3.5$ standard errors); see Table~\ref{tab:sigma}. Third, the naive TTT
estimator converges exactly to the biased limits predicted by
Proposition~\ref{prop:naive} ($+0.921$, $+2$, $+4.5$, $+2$, $+10.42$): its RMSE
is one to two orders of magnitude larger than that of $\alpha_n$ and does not
vanish as $n$ grows. Figure~\ref{Fig1} shows typical realizations at $n=300$: in
every design the parametric estimator $F(x,\alpha_n)$ follows the true df
appreciably more closely than the semiparametric step function $F_n$.

\begin{table}[!htp]
	\centering
	\caption{Monte Carlo bias, SD and RMSE of the proposed estimator $\alpha_n$ and
		of the naive estimator $\alpha_n^{\mathrm{TTT}}$ that ignores left censoring
		($\alpha=1$, $M=50\,000$ replications). Monte Carlo standard errors do not
		exceed $0.0035$ for the bias column, $0.0071$ for the RMSE column, and $0.0012$
		for every cell with $n\ge300$}
	\label{tab:main}
	\small
	\setlength{\tabcolsep}{3.8pt}
	\begin{tabular}{r rrrr rr}
		\toprule
		& \multicolumn{4}{c}{proposed $\alpha_n$ \eqref{eq:alphan}} & \multicolumn{2}{c}{naive $\alpha_n^{\mathrm{TTT}}$}\\
		\cmidrule(lr){2-5}\cmidrule(lr){6-7}
		$n$ & Bias & SD & RMSE & $\sqrt n\,$RMSE & Bias & RMSE\\
		\midrule
		\multicolumn{7}{l}{$(\beta,\theta)=(0.5;\,0.5)$}\\
		30   & 0.0497 & 0.2579 & 0.2627 & 1.439 & 1.0112 & 1.1605\\
		50   & 0.0290 & 0.1907 & 0.1929 & 1.364 & 0.9713 & 1.0551\\
		100  & 0.0157 & 0.1310 & 0.1319 & 1.319 & 0.9448 & 0.9853\\
		300  & 0.0053 & 0.0735 & 0.0737 & 1.277 & 0.9283 & 0.9414\\
		500  & 0.0030 & 0.0571 & 0.0572 & 1.278 & 0.9249 & 0.9327\\
		1000 & 0.0019 & 0.0404 & 0.0405 & 1.280 & 0.9237 & 0.9277\\
		\midrule
		\multicolumn{7}{l}{$(\beta,\theta)=(1,\,1)$}\\
		30   & 0.1295 & 0.4592 & 0.4771 & 2.613 & 2.4092 & 2.9419\\
		50   & 0.0721 & 0.2880 & 0.2969 & 2.099 & 2.2152 & 2.4410\\
		100  & 0.0362 & 0.1838 & 0.1873 & 1.873 & 2.1035 & 2.1966\\
		300  & 0.0120 & 0.0989 & 0.0996 & 1.726 & 2.0333 & 2.0611\\
		500  & 0.0064 & 0.0758 & 0.0761 & 1.702 & 2.0181 & 2.0346\\
		1000 & 0.0035 & 0.0533 & 0.0534 & 1.688 & 2.0086 & 2.0167\\
		\midrule
		\multicolumn{7}{l}{$(\beta,\theta)=(2,\,1)$\ \ (discarded samples: at $n=30$ 0.44\%)}\\
		30   & 0.2133 & 0.6190 & 0.6547 & 3.586 & 5.8550 & 7.3656\\
		50   & 0.1222 & 0.3943 & 0.4128 & 2.919 & 5.2445 & 6.0005\\
		100  & 0.0539 & 0.2224 & 0.2289 & 2.289 & 4.8059 & 5.0415\\
		300  & 0.0190 & 0.1170 & 0.1186 & 2.054 & 4.6004 & 4.6665\\
		500  & 0.0111 & 0.0886 & 0.0893 & 1.998 & 4.5582 & 4.5965\\
		1000 & 0.0056 & 0.0619 & 0.0621 & 1.964 & 4.5286 & 4.5470\\
		\midrule
		\multicolumn{7}{l}{$(\beta,\theta)=(1,\,2)$\ \ (discarded samples: at $n=30$ 0.43\%)}\\
		30   & 0.2338 & 0.7213 & 0.7582 & 4.153 & 2.7236 & 3.6450\\
		50   & 0.1358 & 0.4747 & 0.4938 & 3.492 & 2.4147 & 2.9054\\
		100  & 0.0576 & 0.2562 & 0.2626 & 2.626 & 2.1664 & 2.3235\\
		300  & 0.0191 & 0.1330 & 0.1344 & 2.328 & 2.0528 & 2.0980\\
		500  & 0.0111 & 0.1007 & 0.1013 & 2.266 & 2.0312 & 2.0576\\
		1000 & 0.0052 & 0.0706 & 0.0707 & 2.237 & 2.0135 & 2.0263\\
		\midrule
		\multicolumn{7}{l}{$(\beta,\theta)=(4,\,2)$\ \ (discarded samples: 12.71\%, 3.20\% and 0.10\% at $n=30$, $50$, $100$)}\\
		30   & 0.2201 & 0.6984 & 0.7322 & 4.011 & 12.1606 & 14.0262\\
		50   & 0.2892 & 0.7594 & 0.8126 & 5.746 & 13.8674 & 16.8476\\
		100  & 0.1822 & 0.5808 & 0.6087 & 6.087 & 12.7476 & 15.1480\\
		300  & 0.0496 & 0.2192 & 0.2247 & 3.892 & 11.0261 & 11.4137\\
		500  & 0.0277 & 0.1577 & 0.1601 & 3.581 & 10.7510 & 10.9561\\
		1000 & 0.0152 & 0.1081 & 0.1091 & 3.452 & 10.5942 & 10.6903\\
		\bottomrule
	\end{tabular}
\end{table}

In the extreme design $(4,2)$ only $6.7\%$ of the lifetimes are completely
observed, and at $n=30$ the discarding of degenerate samples ($12.7\%$ of the
replications) induces an appreciable selection effect; even so, already at
$n=300$ the RMSE of $\alpha_n$ is below $0.23$, that is, the parameter is
recovered with useful accuracy from data amounting on average to $20$ complete
observations.

Table~\ref{tab:sigma} compares the limit law of Theorem~\ref{th:clt} with
simulation. For each design the asymptotic standard deviation
$\sigma=(\Var\psi_\alpha)^{1/2}$ was computed from \eqref{eq:IF} by numerical
integration (cross-checked against a direct Monte Carlo evaluation of
$\Var\psi_\alpha$ with $1.6\cdot10^{7}$ replications, agreeing to four decimal
places) and compared with $\sqrt n\cdot\mathrm{SD}(\alpha_n)$. The agreement is
close in every design, and the empirical correlation between
$\sqrt n(\alpha_n-\alpha)$ and the linearized sum tends to one. This confirms
that \eqref{eq:IF} is the correct influence function --- including for the heavy
designs with $\beta\ge1$, in which the normal approximation sets in slowly
(Remark~\ref{rem:beta-ge-1}). The theoretical values of $\sigma$ agree to four
decimal places with the closed formula \eqref{eq:sigmaclosed}.

\begin{table}[!htp]
	\centering
	\caption{Verification of Theorem~\ref{th:clt}: the asymptotic standard deviation
		$\sigma$ obtained from the influence function \eqref{eq:IF}, compared with the
		Monte Carlo $\sqrt n\cdot\mathrm{SD}(\alpha_n)$, together with the correlation
		between $\sqrt n(\alpha_n-\alpha)$ and the linearized sum
		$n^{-1/2}\sum_i\psi_\alpha(Z_i,\Delta_i)$. The simulation uses $M=100\,000$
		replications at $n=1000$ and $M=40\,000$ at $n=5000$. Standard errors of the
		$\sqrt n\cdot\mathrm{SD}$ cells, obtained from the between-block spread, are
		below $0.3\%$ and $0.4\%$ respectively}
	\label{tab:sigma}
	\small
	\begin{tabular}{cc c cc cc}
		\toprule
		& & & \multicolumn{2}{c}{$\sqrt n\cdot\mathrm{SD}(\alpha_n)$} & \multicolumn{2}{c}{correlation}\\
		\cmidrule(lr){4-5}\cmidrule(lr){6-7}
		$(\beta,\theta)$ & $\lambda$ & $\sigma$ (theoretical) & $n=1000$ & $n=5000$ & $n=1000$ & $n=5000$\\
		\midrule
		$(0.5,0.5)$ & 0.667 & 1.2674 & 1.2740 & 1.2668 & 0.9993 & 0.9998\\
		$(1,1)$         & 0.500 & 1.6716 & 1.6843 & 1.6764 & 0.9980 & 0.9996\\
		$(2,1)$         & 0.333 & 1.9335 & 1.9613 & 1.9438 & 0.9966 & 0.9993\\
		$(1,2)$         & 0.500 & 2.1896 & 2.2181 & 2.1963 & 0.9959 & 0.9992\\
		$(4,2)$         & 0.200 & 3.2709 & 3.4022 & 3.3119 & 0.9882 & 0.9978\\
		\bottomrule
	\end{tabular}
\end{table}

Figure~\ref{Fig2}(a) plots $\mathrm{RMSE}(\alpha_n)$ against $n$ on a
double-logarithmic scale. All five designs run parallel to the reference line of
slope $-1/2$, the extreme design $(4,2)$ joining them once $n\ge100$ (below
that, the discarding of degenerate samples distorts the comparison).
Figure~\ref{Fig3} examines the normal approximation itself through normal
quantile--quantile plots of the standardized $\alpha_n$. Theorem~\ref{th:clt}
covers both designs, but the speed of convergence differs sharply. In the light
design $(\beta,\theta)=(0.5;\,0.5)$ the approximation is already excellent at
$n=100$ (skewness $+0.45$) and practically exact at $n=1000$ (skewness
$+0.14$). In the extreme design $(4,2)$ the distribution is strongly
right-skewed at $n=100$ (skewness $+3.84$) and the skewness falls to $+0.58$ at
$n=1000$: the limit law is reached, but slowly. This is the practical content of
Remark~\ref{rem:beta-ge-1}, and a caution against normal-theory confidence
intervals under heavy left censoring in small samples.

\begin{figure}[!htp]
	\begin{center}
		\includegraphics[width=\linewidth]{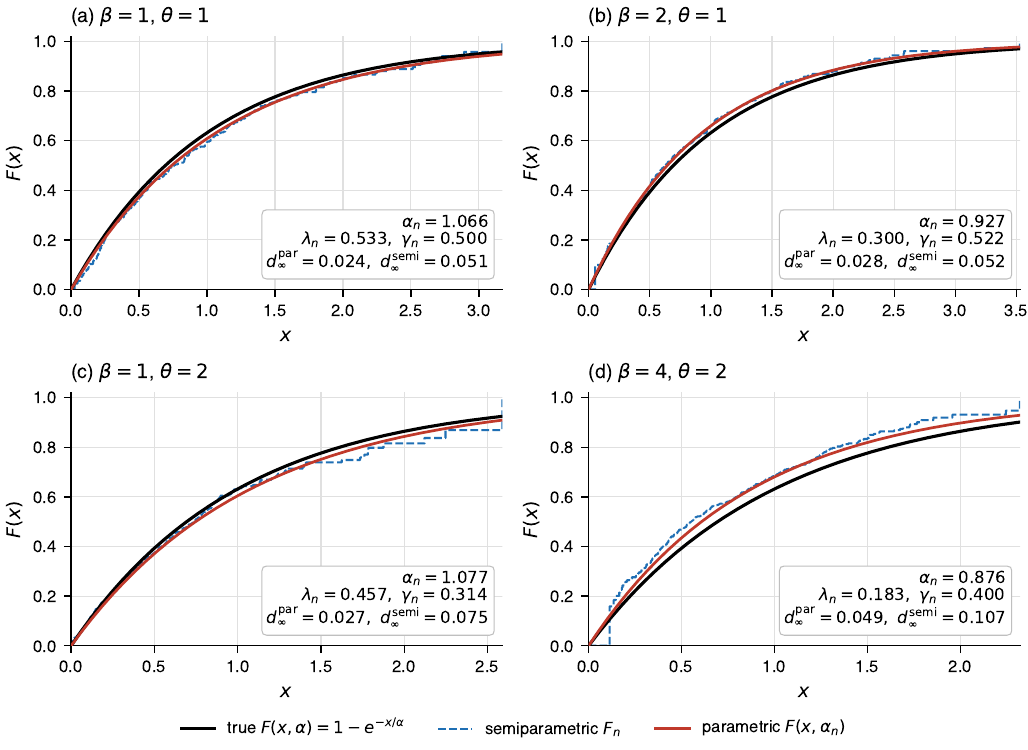}
		\caption{Estimation of the exponential df under two-sided informative
			censorship, $n=300$, $\alpha=1$. The true $F$ (solid black), the semiparametric
			estimator $F_n$ of \eqref{eq:Fn} (dashed blue) and the parametric estimator
			$F(x,\alpha_n)$ of \eqref{eq:alphan} (solid red). To avoid a favourable
			selection, each panel shows the realization attaining the \emph{median} of
			$|\alpha_n-\alpha|$ among $401$ independent replications, that is one of typical
			accuracy. The inset gives the resulting estimates and the two sup-distances
			$d_\infty=\sup_x|\widehat F(x)-F(x)|$}
		\label{Fig1}
	\end{center}
\end{figure}

\subsection{Accuracy of the estimated distribution function: parametric versus
	semiparametric}\label{ssec:df}

Table~\ref{tab:df} compares the plug-in estimator $F(\cdot,\alpha_n)$ with the
semiparametric estimator $F_n$ of \eqref{eq:Fn} in terms of the expected
sup-norm error $\Ex\sup_x|\widehat F(x)-F(x)|$ and the mean integrated squared
error $\mathrm{MISE}=\Ex\int(\widehat F-F)^2\,dF$; Figure~\ref{Fig2}(b) shows
the same comparison graphically. The parametric estimator is superior
throughout. For $n\ge300$ the sup-norm ratio runs from about $2.1$ in the
balanced designs to $3.7$ under heavy right censoring. In the severest design
$(4,2)$ it starts at $1.67$ for $n=50$ and increases with $n$ to $2.86$. The
parametric sup-norm errors agree with the limit law of Corollary~\ref{cor:df},
which predicts
$\Ex\sup_x|\Delta F|\approx e^{-1}\sigma\sqrt{2/\pi}\,n^{-1/2}$: at
$(\beta,\theta)=(0.5;0.5)$ and $n=1000$ this gives $0.01176$ against an observed
$0.01183$; at $(1,1)$ the two values are $0.01552$ and $0.01555$.

\begin{figure}[!htp]
	\begin{center}
		\includegraphics[width=\linewidth]{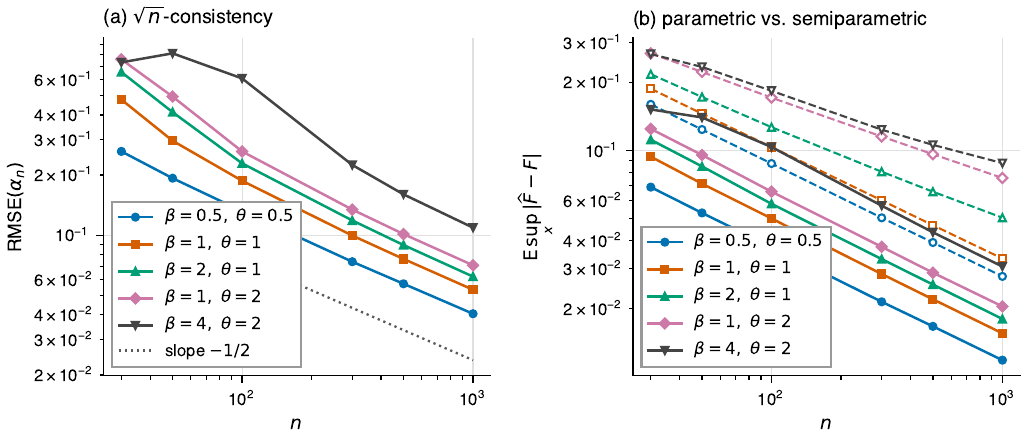}
		\caption{(a) $\mathrm{RMSE}(\alpha_n)$ against $n$ on a double-logarithmic scale
			for the five designs of Table~\ref{tab:designs}, with a reference line of slope
			$-1/2$. (b) For the same designs, the expected sup-norm error of the parametric
			estimator $F(\cdot,\alpha_n)$ (solid, filled markers) and of the semiparametric
			estimator $F_n$ (dashed, open markers)}
		\label{Fig2}
	\end{center}
\end{figure}

\begin{figure}[!htp]
	\begin{center}
		\includegraphics[width=\linewidth]{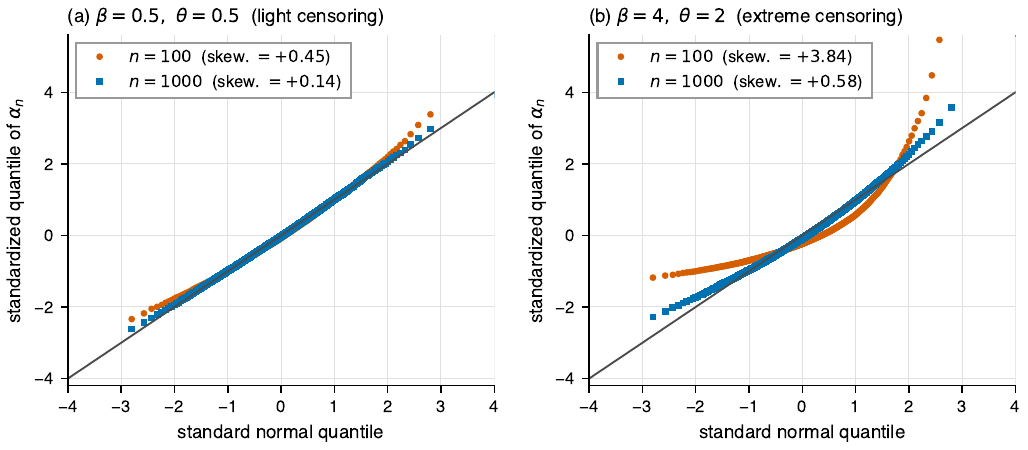}
		\caption{Normal quantile--quantile plots of $\alpha_n$ standardized to zero mean
			and unit variance ($20\,000$ replications per curve). Both designs are covered
			by Theorem~\ref{th:clt}; they differ only in the speed of convergence. Left, a
			light censoring design; right, the extreme design, in which the normal
			approximation is reached considerably more slowly}
		\label{Fig3}
	\end{center}
\end{figure}

\begin{table}[!htp]
	\centering
	\caption{Expected sup-norm error ($\times10^{2}$) and MISE ($\times10^{4}$) of
		the parametric estimator $F(\cdot,\alpha_n)$ and of the semiparametric estimator
		$F_n$, both computed on the same $M=50\,000$ samples. Ratios are
		semiparametric/parametric. Monte Carlo standard errors of the sup-norm cells do
		not exceed $0.051$}
	\label{tab:df}
	\small
	\setlength{\tabcolsep}{4.5pt}
	\begin{tabular}{cc rrr rrr}
		\toprule
		& & \multicolumn{3}{c}{$\Ex\sup_x|\widehat F-F|\times10^{2}$} & \multicolumn{3}{c}{$\mathrm{MISE}\times10^{4}$}\\
		\cmidrule(lr){3-5}\cmidrule(lr){6-8}
		$(\beta,\theta)$ & $n$ & param. & semipar. & ratio & param. & semipar. & ratio\\
		\midrule
		$(0.5,0.5)$ & 50 & 5.29 & 12.36 & 2.34 & 24.2 & 40.1 & 1.66\\
		& 100 & 3.74 & 8.75 & 2.34 & 12.0 & 19.9 & 1.65\\
		& 300 & 2.14 & 5.04 & 2.35 & 3.9 & 6.5 & 1.65\\
		& 1000 & 1.18 & 2.77 & 2.34 & 1.2 & 2.0 & 1.63\\
		\midrule
		$(1,1)$ & 50 & 7.12 & 14.52 & 2.04 & 44.8 & 61.1 & 1.36\\
		& 100 & 5.00 & 10.34 & 2.07 & 21.9 & 30.1 & 1.38\\
		& 300 & 2.84 & 6.01 & 2.11 & 7.0 & 9.6 & 1.39\\
		& 1000 & 1.55 & 3.33 & 2.14 & 2.1 & 2.9 & 1.38\\
		\midrule
		$(2,1)$ & 50 & 8.49 & 17.21 & 2.03 & 65.4 & 89.3 & 1.37\\
		& 100 & 5.80 & 12.66 & 2.18 & 29.8 & 43.3 & 1.45\\
		& 300 & 3.31 & 8.05 & 2.43 & 9.5 & 14.3 & 1.50\\
		& 1000 & 1.80 & 5.04 & 2.80 & 2.8 & 4.3 & 1.54\\
		\midrule
		$(1,2)$ & 50 & 9.54 & 22.20 & 2.33 & 82.0 & 109.6 & 1.34\\
		& 100 & 6.57 & 17.09 & 2.60 & 37.8 & 56.2 & 1.49\\
		& 300 & 3.75 & 11.51 & 3.07 & 12.1 & 19.6 & 1.62\\
		& 1000 & 2.04 & 7.55 & 3.69 & 3.6 & 6.2 & 1.71\\
		\midrule
		$(4,2)$ & 50 & 13.98 & 23.36 & 1.67 & 168.6 & 189.8 & 1.13\\
		& 100 & 10.35 & 18.32 & 1.77 & 99.0 & 113.5 & 1.15\\
		& 300 & 5.70 & 12.34 & 2.17 & 28.7 & 38.0 & 1.32\\
		& 1000 & 3.07 & 8.78 & 2.86 & 8.2 & 12.2 & 1.49\\
		\bottomrule
	\end{tabular}
\end{table}

\section{Robustness under misspecification}\label{sec:robust}

The efficiency documented above is bought at the price of two assumptions: the
exponentiality of $F$ and the informative structure \eqref{eq:model}. This
section quantifies the cost of violating each of them, and shows that for the
departures met in practice that cost is not large.

\subsection{Pseudo-true parameter analysis when the lifetime law is
	misspecified}\label{ssec:pseudo}

Let the censoring mechanism still obey \eqref{eq:model}, but let the true
lifetime df $F_0$ \emph{not} be exponential. Since Theorem~\ref{th:char} depends
on \eqref{eq:model} alone, the constants $\lambda,\gamma$ and their estimators
are unchanged, and the argument of Theorem~\ref{th:cons} shows that
\[
\alpha_n\ \xrightarrow{a.s.}\ \alpha^*=T\bigl(H_0,\lambda,\gamma\bigr)
=\frac{1}{\lambda\gamma}\int_{0}^{1}F_0^{-1}\Bigl(1-\bigl(1-u^{\lambda}\bigr)^{\gamma}\Bigr)\bigl(1-(1-\lambda)\,u^{-\lambda}\bigr)du,
\]
where $H_0=[1-(1-F_0)^{1+\theta}]^{1+\beta}$. In the quasi-maximum-likelihood
terminology of \citet{Wh82}, $\alpha^*$ is the pseudo-true value and
$F(\cdot,\alpha^*)$ is the best exponential approximation to $F_0$ selected by
the estimating equation. As measures of quality we use $|\alpha^*-\Ex X|$
together with the approximation distances
$d_\infty(\alpha^*)=\sup_x|F(x,\alpha^*)-F_0(x)|$ and
$\mathrm{MISE}^*=\int(F(x,\alpha^*)-F_0)^2dF_0$. Moreover, by the triangle
inequality and Corollary~\ref{cor:df},
\begin{equation}
	\sup_x\bigl|F(x,\alpha_n)-F_0(x)\bigr|\le d_\infty(\alpha^*)+\frac{e^{-1}}{\alpha^*}\,|\alpha_n-\alpha^*|\,(1+o_{\Prob}(1)),
\end{equation}
so that the total error splits into a deterministic (small) approximation term
and a stochastic term of order $n^{-1/2}$.

Table~\ref{tab:pseudo} evaluates these quantities by numerical integration for
the design $(\beta,\theta)=(1,1)$ and for Weibull and gamma alternatives
normalized to unit mean. Two features deserve notice. First, the pseudo-true
value is remarkably close to the mean lifetime: for shape departures of
$\pm10\%$, $|\alpha^*-1|<0.9\%$, and even at shapes $0.8$ and $1.2$ it is
$\le2.6\%$ --- the exact Fisher consistency of Theorem~\ref{lem:fisher}
guarantees $\alpha^*(1)=1$ \emph{exactly} in both families, but the order of the
response to a shape error depends on the family. For the Weibull normalized to
unit mean the first-order term vanishes and the effect is quadratic
($\alpha^*-1\approx-0.96\,(k-1)^2$), whereas for the gamma it does not vanish
and the effect is linear, though with a small slope
($\alpha^*-1\approx-0.076\,(k-1)$). In both cases a $10\%$ shape error moves
$\alpha^*$ by less than one per cent. Second, the irreducible approximation
distance $d_\infty(\alpha^*)$ is below $0.04$ for shape departures of $\pm10\%$
--- a level comparable with the \emph{stochastic} error of the semiparametric
estimator at $n\approx750$.

\begin{table}[!htp]
	\centering
	\caption{Pseudo-true parameter and best-exponential approximation distances for
		misspecified lifetime laws (unit mean, censoring design $\beta=\theta=1$)}
	\label{tab:pseudo}
	\small
	\begin{tabular}{l cccc}
		\toprule
		true $F_0$ & shape & $\alpha^*$ & $d_\infty(\alpha^*)$ & $\mathrm{MISE}^*\times10^4$\\
		\midrule
		Weibull & 0.80 & 0.9747 & 0.0794 & 26.4\\
		Weibull & 0.90 & 0.9945 & 0.0379 & 6.0\\
		Weibull & 0.95 & 0.9987 & 0.0185 & 1.4\\
		Weibull & 1.05 & 0.9989 & 0.0178 & 1.3\\
		Weibull & 1.10 & 0.9958 & 0.0350 & 5.2\\
		Weibull & 1.20 & 0.9850 & 0.0676 & 19.5\\
		Gamma   & 0.90 & 1.0063 & 0.0247 & 2.5\\
		Gamma   & 1.10 & 0.9915 & 0.0226 & 2.1\\
		\bottomrule
	\end{tabular}
\end{table}

\subsection{Finite-sample robustness and comparison with the semiparametric
	estimator}

Table~\ref{tab:robust} reports, for the most relevant alternatives (Weibull and
gamma with shapes $0.9$ and $1.1$), the empirical bias and RMSE of $\alpha_n$
\emph{relative to the pseudo-true value $\alpha^*$}, together with the expected
sup-norm errors of $F(\cdot,\alpha_n)$ and of the semiparametric estimator $F_n$
(which is correctly specified, assuming no parametric form). Both are measured
against the true $F_0$, with $M=20\,000$ replications and design $(1,1)$. The
convergence $\alpha_n\to\alpha^*$ predicted by quasi-likelihood theory is
clearly visible (the bias column vanishes, the RMSE column decays like
$n^{-1/2}$). More important for practice is that the \emph{misspecified
	parametric estimator beats the correctly specified semiparametric one} in
sup-norm up to sample sizes between $500$ and $1000$ for $\pm10\%$ Weibull
departures, and between $1000$ and $2000$ for gamma departures. Only beyond
those sample sizes does the irreducible approximation bias $d_\infty(\alpha^*)$
overtake the advantage in variance. Since both estimators are computed on the
same samples, the comparison is paired, and the last column of
Table~\ref{tab:robust} shows that the advantage before the crossing --- and the
deficit after it --- lies well outside Monte Carlo error: the paired
$z$-statistics range from $11$ to $303$ in absolute value. Figure~\ref{Fig4}
plots the two error curves together with the asymptotic lower bound
$d_\infty(\alpha^*)$, making the crossing explicit.

\begin{table}[!htp]
	\centering
	\caption{Robustness study, design $\beta=\theta=1$, $M=20\,000$ replications
		($10\,000$ at $n=2000$): bias and RMSE of $\alpha_n$ relative to the pseudo-true
		$\alpha^{*}$, and expected sup-norm errors relative to the true $F_0$
		($\times10^{2}$), the columns headed ``param.'' and ``semipar.'' referring to
		$F(\cdot,\alpha_n)$ and $F_n$ respectively. The last column is the \emph{paired}
		mean difference between
		the two estimators, computed replication by replication; a positive value means
		the parametric estimator is closer to $F_0$. Monte Carlo standard errors do not
		exceed $0.0020$ (bias), $0.0033$ (RMSE) and $0.026\times10^{-2}$ (paired
		difference); pairing makes the variance of the difference much smaller, since
		both estimators are computed on the same samples. Every cell of the last column
		lies between $11$ and $303$ standard errors from zero}
	\label{tab:robust}
	\footnotesize
	\setlength{\tabcolsep}{3.5pt}
	\begin{tabular}{l r rr rr r}
		\toprule
		& & \multicolumn{2}{c}{$\alpha_n$ vs.\ $\alpha^{*}$} & \multicolumn{3}{c}{$\Ex\sup|\widehat F-F_0|\times10^{2}$}\\
		\cmidrule(lr){3-4}\cmidrule(lr){5-7}
		true $F_0$ & $n$ & Bias & RMSE & param. & semipar. & paired\\
		\midrule
		Weibull(0.9) & 50 & $+0.0633$ & 0.2906 & 8.28 & 14.52 & $+6.24$\\
		& 100 & $+0.0316$ & 0.1851 & 6.44 & 10.30 & $+3.86$\\
		& 300 & $+0.0098$ & 0.1017 & 4.81 & 6.01 & $+1.20$\\
		& 500 & $+0.0066$ & 0.0770 & 4.39 & 4.67 & $+0.28$\\
		& 1000 & $+0.0028$ & 0.0535 & 4.03 & 3.31 & $-0.72$\\
		& 2000 & $+0.0020$ & 0.0387 & 3.90 & 2.38 & $-1.52$\\
		\midrule
		Weibull(1.1) & 50 & $+0.0800$ & 0.2927 & 8.26 & 14.48 & $+6.22$\\
		& 100 & $+0.0381$ & 0.1841 & 6.38 & 10.35 & $+3.97$\\
		& 300 & $+0.0124$ & 0.0989 & 4.61 & 6.02 & $+1.41$\\
		& 500 & $+0.0068$ & 0.0753 & 4.16 & 4.68 & $+0.52$\\
		& 1000 & $+0.0042$ & 0.0523 & 3.76 & 3.32 & $-0.44$\\
		& 2000 & $+0.0022$ & 0.0369 & 3.59 & 2.36 & $-1.22$\\
		\midrule
		Gamma(0.9) & 50 & $+0.0664$ & 0.2935 & 7.62 & 14.56 & $+6.94$\\
		& 100 & $+0.0315$ & 0.1852 & 5.61 & 10.31 & $+4.70$\\
		& 300 & $+0.0113$ & 0.1013 & 3.83 & 6.02 & $+2.18$\\
		& 500 & $+0.0056$ & 0.0774 & 3.32 & 4.69 & $+1.37$\\
		& 1000 & $+0.0035$ & 0.0537 & 2.89 & 3.31 & $+0.43$\\
		& 2000 & $+0.0020$ & 0.0383 & 2.66 & 2.37 & $-0.29$\\
		\midrule
		Gamma(1.1) & 50 & $+0.0740$ & 0.2910 & 7.66 & 14.54 & $+6.88$\\
		& 100 & $+0.0380$ & 0.1832 & 5.62 & 10.31 & $+4.70$\\
		& 300 & $+0.0128$ & 0.0985 & 3.71 & 6.01 & $+2.30$\\
		& 500 & $+0.0076$ & 0.0750 & 3.18 & 4.68 & $+1.51$\\
		& 1000 & $+0.0040$ & 0.0531 & 2.72 & 3.33 & $+0.61$\\
		& 2000 & $+0.0020$ & 0.0373 & 2.46 & 2.37 & $-0.08$\\
		\bottomrule
	\end{tabular}
\end{table}

\begin{figure}[!htp]
	\begin{center}
		\includegraphics[width=\linewidth]{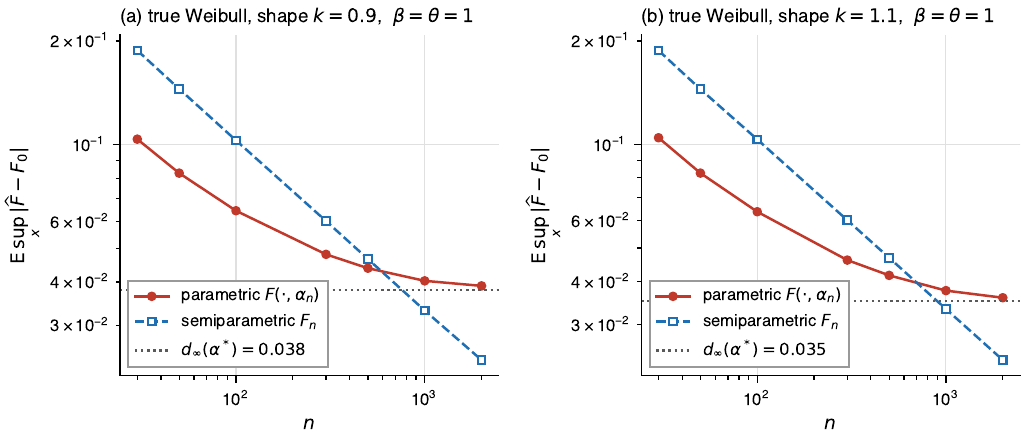}
		\caption{Expected sup-norm error against the true $F_0$ under a misspecified
			lifetime law, design $\beta=\theta=1$: the parametric estimator
			$F(\cdot,\alpha_n)$ (solid) and the semiparametric estimator $F_n$ (dashed),
			with the irreducible approximation distance $d_\infty(\alpha^*)$ shown as a
			horizontal line. The misspecified parametric estimator is the more accurate of
			the two until the two curves cross}
		\label{Fig4}
	\end{center}
\end{figure}

\subsection{Violation of the informativeness condition}\label{ssec:kviol}

The structural condition $K=N^{\beta}$ is of a different nature. It is precisely
what makes the censoring \emph{informative}, and by Theorem~\ref{th:char} it is
\emph{testable}, since the independence of $Z_i$ and $\Delta_i$ can be checked
by standard tests (compare \citealp{Ab94, Cs89}). One qualification is in order.
Theorem~\ref{th:char} makes independence equivalent to the \emph{whole} of
\eqref{eq:model}, not to $K=N^{\beta}$ alone. An independence test therefore
examines both halves at once, and a rejection does not indicate which half has
failed. This matters for $\alpha_n$, because by Theorem~\ref{th:gfree} only the
entry half affects its consistency; in that sense the diagnostic is
conservative, a point we return to in Section~\ref{ssec:gviol}.

To quantify the sensitivity of $\alpha_n$ we violated the informative
left-censoring law: with probability $1-\varepsilon$ the entry time $L_i$ is
drawn from the informative law $K=N$ ($\beta=1$), and with probability
$\varepsilon$ from a uniform law on $(0;0.797)$ calibrated so that the depth of
left censoring $p^{(0)}=1/2$ is unchanged ($X\sim\mathrm{Exp}(1)$, $\theta=1$).
Table~\ref{tab:contam} shows that the resulting asymptotic bias is almost linear
in the degree of contamination, $\alpha^{\dagger}-\alpha\approx-0.21\,\varepsilon$:
a contamination of $5$--$10\%$ produces an asymptotic bias of only $1$--$2\%$,
negligible against the sampling error at $n=300$ (RMSE $\approx0.10$). Even the
harshest violations considered --- $L$ entirely uniform, or Weibull with shape
$2$, at a left-censoring depth of $50\%$ --- give the limits $0.786$ and
$0.743$, that is errors of $21$--$26\%$: appreciable, but bounded and, for a
sufficiently large sample, \emph{detectable} in advance by the independence
diagnostic of Theorem~\ref{th:char}.

\begin{table}[!htp]
	\centering
	\caption{Violation of the informativeness condition ($X\sim\mathrm{Exp}(1)$,
		$\theta=1$, the depth of left censoring held at $p^{(0)}=1/2$): asymptotic
		(pseudo-limit) bias of $\alpha_n$ and its finite-sample behaviour at $n=300$
		($M=20\,000$ replications; Monte Carlo standard errors at most $0.0007$)}
	\label{tab:contam}
	\small
	\begin{tabular}{c ccc}
		\toprule
		contamination $\varepsilon$ & asymptotic bias & bias at $n=300$ & RMSE at $n=300$\\
		\midrule
		0    & $+0.0006$ & $+0.0118$ & 0.0995\\
		0.05 & $-0.0099$ & $-0.0003$ & 0.0976\\
		0.10 & $-0.0202$ & $-0.0086$ & 0.0981\\
		0.25 & $-0.0509$ & $-0.0397$ & 0.1031\\
		0.50 & $-0.1035$ & $-0.0939$ & 0.1312\\
		1.00 & $-0.2138$ & $-0.2034$ & 0.2193\\
		\bottomrule
	\end{tabular}
\end{table}

\emph{How reliable is the diagnostic?} The claim that the condition is
``testable'' has practical content only if the test has adequate power, so we
measured it. Table~\ref{tab:power} gives the rejection frequency at level
$0.05$ against the contamination alternative above. Both tests hold their size
correctly (the rows $\varepsilon=0$), but at small sample sizes the power is
low: at $n=30$ the coarsest violation ($\varepsilon=1$), which produces a
$-21\%$ bias, is detected in only a quarter of cases, and a $-10\%$ bias
($\varepsilon=0.5$) in fewer than one case in ten. From $n\ge300$ the power
becomes satisfactory for serious violations ($0.55$ at $\varepsilon=0.5$ and
$0.996$ at $\varepsilon=1$ for the Kruskal--Wallis test). Mild contamination
($\varepsilon=0.25$) may pass unnoticed, but the $-5\%$ bias it produces is
smaller than the sampling error at $n=300$ (RMSE $\approx0.10$) --- that is, the
power of the test is directed exactly at the violations that matter.

Two practical conclusions follow. (i) A non-rejection by the diagnostic is not
evidence that the model holds when $n\lesssim100$. For such samples the
condition $K=N^{\beta}$ must rest on substantive grounds --- usually on the fact
that the entry procedure is governed by the study design --- and not on the
diagnostic. (ii) When all three groups are non-empty the Kruskal--Wallis test is
appreciably more powerful than the correlation test (for instance $0.748$
against $0.506$ at $n=100$, $\varepsilon=1$), and we recommend it. When
$\beta=0$ (the 6-MP data of Section~\ref{sec:data}) the third group is empty and
the correlation test remains the natural choice.

The figures in Table~\ref{tab:power} should be read as a \emph{lower bound}.
Both tests are omnibus, examining the whole of \eqref{eq:model} at once. The
criterion developed in \citet{AM26k} splits the hypothesis into two separately
testable halves and provides a statistic aimed \emph{specifically} at the entry
half; against the same contamination alternative it is appreciably more
powerful --- a power of $0.61$ at $n=200$ and $\varepsilon=0.5$, against $0.40$
for Kruskal--Wallis on the same samples. It also indicates which half of
\eqref{eq:model} the rejection concerns, which by Theorem~\ref{th:gfree} is a
decisive distinction for $\alpha_n$.

\begin{table}[!htp]
	\centering
	\caption{Power of the independence diagnostic of Theorem~\ref{th:char}:
		rejection frequency at level $0.05$ against the contamination alternative of
		Table~\ref{tab:contam} ($M=4000$ replications, $999$ permutations for the
		permutation test). The rows $\varepsilon=0$ give the actual size of the test.
		Monte Carlo standard errors are at most $0.008$}
	\label{tab:power}
	\small
	\begin{tabular}{c ccccc}
		\toprule
		contamination $\varepsilon$ & $n=30$ & $n=50$ & $n=100$ & $n=300$ & $n=1000$\\
		\midrule
		\multicolumn{6}{l}{\emph{(a) correlation permutation test $|r(Z,\delta^{(1)})|$}}\\
		0    & 0.046 & 0.056 & 0.050 & 0.052 & 0.055\\
		0.25 & 0.064 & 0.075 & 0.080 & 0.126 & 0.263\\
		0.50 & 0.099 & 0.116 & 0.151 & 0.332 & 0.765\\
		1.00 & 0.234 & 0.332 & 0.506 & 0.902 & 1.000\\
		\midrule
		\multicolumn{6}{l}{\emph{(b) Kruskal--Wallis: $Z\mid\delta^{(m)}=1$, $m=0,1,2$}}\\
		0    & 0.043 & 0.046 & 0.047 & 0.050 & 0.048\\
		0.25 & 0.059 & 0.066 & 0.097 & 0.164 & 0.438\\
		0.50 & 0.096 & 0.129 & 0.211 & 0.548 & 0.976\\
		1.00 & 0.281 & 0.440 & 0.748 & 0.996 & 1.000\\
		\bottomrule
	\end{tabular}
\end{table}

\subsection{Violation of the right-censoring condition}\label{ssec:gviol}

Theorem~\ref{th:gfree} predicts that replacing $G$ by an arbitrary law leaves
$\alpha_n$ consistent. We checked this by simulation. Keeping
$X\sim\mathrm{Exp}(1)$ and $K=N^{\beta}$ ($\beta=1$) exactly, we drew $Y$ from
laws that are \emph{not} powers of $1-F$, each calibrated so that
$\Prob(X<Y)=1/2$ as in the design $\theta=1$. Estimating the limit of
$\alpha_n$ from a single sample of $3\cdot10^{6}$ observations gives $1.0009$
for uniform $G$, $1.0004$ for Weibull $G$ with shape $2$ and $1.0006$ for
lognormal $G$. A systematic $\varepsilon$-contamination of $G$ by a uniform law
gives limits between $0.9991$ and $1.0006$ for every $\varepsilon$ in $[0,1]$.
All of these agree with $\alpha=1$ to within Monte Carlo error --- in sharp
contrast with the asymptotic biases of $21$--$26\%$ produced in
Section~\ref{ssec:kviol} by violating the \emph{left} condition.

Here Theorems~\ref{th:char} and~\ref{th:gfree} may appear to contradict one
another, and the appearance should be dispelled. When the right half is
violated, \eqref{eq:model} does not hold in full, so by Theorem~\ref{th:char}
$Z$ and $\Delta$ \emph{must} be dependent; and yet Theorem~\ref{th:gfree} leaves
$\alpha_n$ consistent. Both facts are visible side by side on one sample. At
$n=4\cdot10^{6}$, for the three alternative laws $G$ above, the largest
between-group difference of $\Ex[Z\mid\Delta]$ is $0.056$--$0.211$, that is
$72$--$402$ times its own Monte Carlo standard error: the dependence \emph{is}
there, and it is not small in any statistical sense, only small in magnitude. On
those very samples $\alpha_n$ departs from the target by at most $0.0013$. The
reason there is no contradiction is visible in the proof of
Theorem~\ref{th:gfree}: it never invokes $Z\perp\Delta$, but integrates the
joint law of $(L,X,Y)$ directly, and what it requires of that law is a single
moment identity rather than the whole independence structure. The estimator asks
less of the model than the model asserts, and therefore survives when the model
fails.

One consequence deserves emphasis for practice. The diagnostic of
Theorem~\ref{th:char} tests the \emph{whole} of \eqref{eq:model}, and a wrong
$G$ does genuinely destroy the independence of $Z$ and $\Delta$. In the
experiment described above the permutation test rejects for the uniform
alternative ($p=0.005$ at $n=300$) but not for the Weibull and lognormal
alternatives ($p=0.77$ in each). A rejection therefore does not by itself
\emph{invalidate} $\alpha_n$: one must go on to ask \emph{which} half of
\eqref{eq:model} has failed, since only the entry half matters. In that sense
the diagnostic is conservative.

The price paid for a wrong $G$ is in the variance alone. At $n=300$ with
$M=20\,000$, the RMSE of $\alpha_n$ rises from $0.0991$ under the model to
$0.106$, $0.109$ and $0.120$ for the lognormal, uniform and Weibull
alternatives respectively --- in the worst case an increase of about $21\%$ in
RMSE, with no bias. Together with Section~\ref{ssec:kviol} this yields a clear
and asymmetric picture: \emph{the entry mechanism must be informative; the
	mechanism of exit from observation need not be.}

\section{Applications to real data}\label{sec:data}

\subsection{Maintenance therapy with 6-MP in acute leukaemia (the boundary case
	$\beta=0$)}\label{ssec:sixmp}

The remission-duration data of the 6-mercaptopurine (6-MP) arm of the clinical
trial of \citet{Fr63}, popularized by \citet{Ge65}, are a classical benchmark of
right-censored survival analysis \citep{KM03}: $n=21$ patients with remission
lengths (in weeks)
\[
\begin{array}{l}
	6,\;6,\;6,\;6^{+},\;7,\;9^{+},\;10,\;10^{+},\;11^{+},\;13,\;16,\;17^{+},\\
	19^{+},\;20^{+},\;22,\;23,\;25^{+},\;32^{+},\;32^{+},\;34^{+},\;35^{+} .
\end{array}
\]
($^{+}$ denotes censoring; $9$ relapses and $12$ censored). There is no left
censoring here, so the model applies with $\beta=0$, $\lambda_n=1$ and
$\gamma_n=9/21=0.429$, and by Remark~\ref{rem:KG} the estimator
\eqref{eq:alphan} reduces to the total time on test:
\[
\alpha_n=\frac{359}{9}=39.89\ \text{weeks},
\]
reproducing exactly the classical censored-exponential analysis of these data.
The implied median remission is $\alpha_n\log2=27.7$ weeks against a
Kaplan--Meier median of $23$ weeks. The fitted exponential survival curve agrees
with the Kaplan--Meier estimator to within $0.05$ at every event time but the
last (at $t=23$ the Kaplan--Meier risk set is nearly exhausted and the
discrepancy is $0.11$). The informativeness diagnostic of
Theorem~\ref{th:char} --- a permutation test based on the sample correlation
between $Z_i$ and $\delta_i$ ($r=-0.44$) --- gives $p=0.046$: for these data the
Koziol--Green hypothesis is borderline, a known property \citep{CF98}. For
$\alpha_n$, however, this is immaterial. Here $\beta=0$, so the entry condition
$K=N^{\beta}$ holds trivially ($K\equiv1$) and consistency is preserved
\emph{whatever} the right-censoring law: the borderline Koziol--Green $p$-value
concerns the semiparametric estimator \eqref{eq:Fn}, not \eqref{eq:alphan}.
Since $\beta=0$ is the boundary of the range $\beta>0$ stated in
Theorem~\ref{th:gfree}, we show this directly: by the constant hazard
$f=(1-F)/\alpha$, for arbitrary $G$,
\[
\Prob(X<Y)=\int_{0}^{\infty}\bigl(1-G\bigr)f\,dt
=\frac1\alpha\int_{0}^{\infty}\bigl(1-F\bigr)\bigl(1-G\bigr)dt
=\frac{\Ex\min(X,Y)}{\alpha},
\]
so that $T=\Ex\min(X,Y)/\Prob(X<Y)=\alpha$. This is the classical fact that,
under independent right censoring, the exponential maximum likelihood estimator
requires no condition on $G$ whatever; Theorem~\ref{th:gfree} generalizes it to
$\beta>0$. The exponential fit itself also remains satisfactory.

\subsection{Insulating-fluid breakdown data under artificial two-sided
	censorship}\label{ssec:fluid}

Nelson's accelerated test of insulating fluid \citep{Ne72} gives $n=19$ times
(in minutes) to dielectric breakdown at $34\,$kV:
\[
\begin{array}{l}
	0.19,\ 0.78,\ 0.96,\ 1.31,\ 2.78,\ 3.16,\ 4.15,\ 4.67,\ 4.85,\ 6.50,\\
	7.35,\ 8.01,\ 8.27,\ 12.06,\ 31.75,\ 32.52,\ 33.91,\ 36.71,\ 72.89 .
\end{array}
\]
Engineering theory suggests an exponential law of breakdown at fixed voltage,
and these data are a standard textbook illustration of exponential inference
\citep{La03}. The maximum likelihood estimator from the complete data is
$\widehat\alpha=\bar x=14.36$ minutes. A Weibull fit gives shape
$\widehat k=0.77$ (with $95\%$ \emph{profile-likelihood} interval
$[0.53;\ 1.06]$; at $n=19$ the likelihood is appreciably skewed in $k$, so a
Wald interval is unreliable here), and a likelihood ratio test of exponentiality
is not rejected ($p=0.116$), although the point estimate suggests the mild
departure regime studied in Section~\ref{sec:robust}. Because the
Kolmogorov--Smirnov statistic is computed against an \emph{estimated}
exponential, its null law is non-standard; a Lilliefors-type parametric
bootstrap gives $p=0.040$ --- rejection at the $5\%$ level, whereas the Weibull
likelihood ratio test does not reject. The disagreement between the two tests
shows that the exponential shape is borderline here, and it is the reason why
below the target of the estimator is taken to be the pseudo-true $\alpha^{*}$
rather than $\widehat\alpha$.

Since the complete sample is available, we can impose \emph{artificial}
two-sided informative censorship with known ground truth: holding the $19$
actual breakdown times fixed, we generate $Y_i$ and $L_i$ from \eqref{eq:model}
(taking $\widehat\alpha$ for $\alpha$), form the censored sample
$(Z_i,\Delta_i)$ and compute $\alpha_n$; this is repeated $5\cdot10^{4}$ times
for each design. Two benchmarks are relevant here and should not be confused.
The complete-data estimate $\widehat\alpha=14.36$ is the value one would report
without censoring; but since the breakdown times are not exactly exponential,
the quantity the estimator actually targets is the pseudo-true value
$\alpha^{*}$ of Section~\ref{ssec:pseudo} for the empirical df of the $19$
observations, which depends on the design and is given in
Table~\ref{tab:fluid}. These differ across the four designs by $-4.8\%$,
$+7.8\%$, $-23.6\%$ and $-7.1\%$ --- a direct manifestation of the mild
Weibull-type departure quantified in Section~\ref{sec:robust}, amplified when
the design emphasizes the right tail.

Measured against the appropriate target $\alpha^{*}$, the median of $\alpha_n$
attains an accuracy of $2.7\%$, $4.0\%$ and $6.6\%$ in the three moderate
designs and $10.0\%$ in the extreme design $(4,2)$, where $82\%$ of the
observations are left-censored and only $8\%$ of the breakdown times are
observed exactly, so that $92\%$ are incomplete. The mean of $\alpha_n$ lies
everywhere well above the median: at $n=19$ the sampling distribution is
strongly right-skewed --- exactly as Figure~\ref{Fig3}(b) documents for heavy
censoring at small $n$ --- so at this sample size the median is the appropriate
summary. The sampling variability is appreciable --- as it must be with $n=19$
and up to $92\%$ incompleteness --- but the location of the estimator sits
firmly on the value it is designed to estimate.

\begin{table}[!htp]
	\centering
	\caption{Artificial two-sided informative censorship imposed on the $19$
		insulating-fluid breakdown times (complete-data $\widehat\alpha=14.36$;
		$5\cdot10^{4}$ censoring replications per design). Here $\alpha^{*}$ is the
		pseudo-true target of Section~\ref{ssec:pseudo} for the empirical df of the $19$
		observations, and $\bar p^{(0)},\bar p^{(1)}$ are the means of the realized
		censoring pattern. Monte Carlo standard errors of the mean and median are at
		most $0.05$. Replications with no complete observation are discarded (their
		proportions are $0.6\%$, $7.0\%$, $1.8\%$ and $34.4\%$ respectively)}
	\label{tab:fluid}
	\small
	\begin{tabular}{cc c ccc c}
		\toprule
		$(\beta,\theta)$ & $\bar p^{(0)}$ / $\bar p^{(1)}$ & $\alpha^{*}$ & mean$(\alpha_n)$ & median$(\alpha_n)$ & SD$(\alpha_n)$ & $|$median$-\alpha^{*}|/\alpha^{*}$\\
		\midrule
		$(1,1)$ & 0.56 / 0.22 & 13.67 & 15.89 & 14.04 & 7.72 & 2.7\%\\
		$(2,1)$ & 0.72 / 0.13 & 15.47 & 18.75 & 16.09 & 9.64 & 4.0\%\\
		$(1,2)$ & 0.55 / 0.18 & 10.98 & 13.84 & 11.71 & 7.70 & 6.6\%\\
		$(4,2)$ & 0.82 / 0.08 & 13.34 & 13.67 & 12.01 & 7.10 & 10.0\%\\
		\bottomrule
	\end{tabular}
\end{table}

\subsection{Genuinely left-censored data: AIDS in intravenous drug
	users}\label{ssec:ivaids}

In both examples above left censoring was either absent ($\beta=0$) or imposed
artificially. In the following data set it is \emph{genuine}. Studied by
\citet{JG11} and distributed through the package of \citet{dblcens}, these data
concern $232$ intravenous drug users enrolled in a detoxification programme in
Badalona, Spain; $Z$ is the time from first injection to a diagnosis of AIDS.
Of these, $14$ are \emph{left}-censored (having died of AIDS without a prior
diagnosis), $82$ are completely observed and $136$ are right-censored, so that
$p^{(0)}=6.0\%$, $\widehat\beta=0.064$ and $\widehat\theta=1.66$.

\emph{First, the condition is tested.} The criterion of \citet{AM26k} splits
\eqref{eq:model} into two separately testable halves, and under permutation
calibration the outcome is clear-cut: the entry half ($K=N^{\beta}$) is not
rejected ($p\in[0.24;\,0.42]$ for all three statistics), whereas the
right-censoring half ($1-G=(1-F)^{\theta}$) is decisively rejected
($p\le0.0008$). The first conclusion must be read with care: since there are
only $14$ left-censored observations, the test of the entry half has low power
--- in this configuration ($n=232$, $p^{(0)}=6.0\%$) it rejects with probability
only $0.62$ even when the entry law is replaced \emph{entirely}, and with
probability $0.19$ under half contamination. A non-rejection therefore does not
confirm the entry model; it only indicates the absence of evidence against it.
The rejection, by contrast, is expected and interpretable: right censoring here
is administrative --- the end of the observation period --- and there is no
reason for it to be a power of the survival function; indeed the median of the
complete observations is $10.4$ years against $13.0$ years for the
right-censored ones, whereas the Koziol--Green condition requires $Z$ and
$\delta^{(1)}$ to be independent.

\emph{It is here that the value of a split diagnosis becomes apparent.} A
monolithic test would say only ``the model is rejected'' and the analysis would
stop there. By Theorem~\ref{th:gfree}, however, only the entry half is needed
for the estimator \eqref{eq:alphan}: whatever $G$ may be, $\alpha_n$ remains
exactly Fisher-consistent and strongly consistent. The semiparametric estimator
\eqref{eq:Fn}, by contrast, rests on the representation
$1-F=(1-H^{\lambda})^{\gamma}$ and requires \emph{both} halves --- so on these
data it is invalid. A rejected model therefore does not halt the analysis; it
indicates which estimator may be used.

The estimator gives $\alpha_n=27.5$ years, whereas the naive TTT estimator that
ignores left censoring gives $32.7$ years --- higher, exactly as
Proposition~\ref{prop:naive} predicts. The quantitative agreement is
approximate: at the depth $\widehat\beta=0.064$ the proposition predicts an
inflation by a factor of $1.11$ ($30.5$ years) against an observed factor of
$1.19$. The discrepancy is expected, since Proposition~\ref{prop:naive} requires
both halves of \eqref{eq:model} and an exponential $F$, neither of which holds
here. The ten-year AIDS-free probability comes out as $0.695$; the nonparametric
analysis of \citet{JG11} reports about $0.70$.

\emph{Exponentiality, on the other hand, is not supported.} The nonparametric
maximum likelihood estimator for doubly censored data (a self-consistency
algorithm) puts the median at $15.5$ years --- in agreement with the value of
\citet{JG11} --- but at a sup-norm distance of $0.126$ from the fitted
exponential df. Passing to the family $A(x)=x^{k}$ of
Proposition~\ref{prop:phclass} reduces this distance to $0.052$ at
$k\approx1.65$ (scale $21.0$ years), which points to an increasing hazard --- as
one would expect for HIV$\to$AIDS progression. These data therefore demonstrate
at once the value of Theorem~\ref{th:gfree} and the necessity of
Section~\ref{ssec:general}: there is no evidence against the entry half --- a
statement that must be read together with the power caveat above --- and in that
situation the estimator retains its force; but the exponential shape of the
lifetime law, the departure measured in Section~\ref{sec:robust}, is not
acceptable here, and $\alpha_n$ should be read as the pseudo-true parameter of
Section~\ref{ssec:pseudo}.

\section{Discussion}\label{sec:disc}

The results assembled above single out the estimator \eqref{eq:alphan} --- at
once practical and theoretically grounded --- as a tool for lifetime estimation
under simultaneous left and right censorship, and they make clear where its
advantages come from.

\textbf{Closed form and computational transparency.} Unlike nonparametric
maximum likelihood for doubly censored data, which requires iterative
self-consistency algorithms, and unlike parametric likelihoods under
non-informative two-sided censoring, which require numerical maximization,
$\alpha_n$ is an explicit $L$-statistic computed in a single pass over the
ordered sample. This simplicity is supplied precisely by the informative
structure \eqref{eq:model}, which condenses all the information about the
censoring mechanism into the two frequencies $p_n^{(0)},p_n^{(1)}$. The
convenience is almost free: as Section~\ref{ssec:eff} shows, the asymptotic
variance of $\alpha_n$ lies only $0.6\%$--$1.4\%$ above the information bound of
the model in the designs considered ($98.6\%$--$99.4\%$ efficiency), the loss
over a wider sweep does not exceed $3.8\%$, and it tends to zero in the limit of
vanishing left censoring. Full numerical maximization would add almost nothing.

\textbf{Efficiency.} Within the model, $\alpha_n$ is exactly Fisher-consistent
(Theorem~\ref{lem:fisher}), $\sqrt n$-consistent in every design we examined,
and asymptotically normal with explicit closed-form constants for all
$\beta,\theta>0$ (Theorem~\ref{th:clt}); and its plug-in df estimator is
superior to the semiparametric power-type estimator and, under informativeness,
to product-limit-type estimators \citep{Cs88, CF98, Pa99} --- better by a factor
of $2.1$--$3.7$ in expected sup-norm for $n\ge300$ (at least $1.7$ at smaller
sizes; Table~\ref{tab:df}). The asymptotic theory moreover gives an accurate
description already at moderate $n$: the influence function \eqref{eq:IF}
reproduces the simulated $\sqrt n\cdot\mathrm{SD}(\alpha_n)$ to within
$0.5$--$4\%$ at $n=1000$ (Table~\ref{tab:sigma}). The deviation is positive in
all five designs and grows with the depth of censoring --- a systematic
finite-sample excess that nonetheless vanishes with $n$: at $n=5000$ it has
disappeared into Monte Carlo error in four designs (from $-0.04\%$ to
$+0.53\%$, all within $1.4$ standard errors), remaining at $1.3\%$ ($3.5$
standard errors) only in the severest design $(4,2)$. Corollary~\ref{cor:df}
predicts the observed sup-norm errors at $n=1000$ to a relative accuracy of
$0.6\%$ in four designs and $1.0\%$ in the severest design $(4,2)$. Remarkably,
this accuracy also holds for the heavy designs with $\beta\ge1$, in which the
normal approximation sets in slowly.

\textbf{The scope of the method is exact.} The closed form is not an accident of
the exponential. Section~\ref{ssec:general} derives the likelihood equation for
an arbitrary family (Proposition~\ref{prop:general}) and shows that the weight
$1-(1-\lambda)H^{-\lambda}$ in \eqref{eq:alphan} comes from the model, not from
the family. The equation admits an explicit solution \emph{precisely} on the
proportional hazards class (Propositions~\ref{prop:phclass}
and~\ref{prop:converse}), and within that class every result reported here ---
Fisher consistency, asymptotic normality, \eqref{eq:sigmaclosed}, the
efficiency, Theorem~\ref{th:gfree} --- is preserved verbatim at every $n$,
because the distribution of the \emph{transformed} observation $A(Z)$ coincides
with that of $Z$ in the exponential case $A(x)=x$. This must be read both ways:
it is an extension on the one hand and a limitation on the other, since $A$ is
required to be known exactly and outside the class there is no explicit solution
(Remark~\ref{rem:family}).

\textbf{The two-sided model is not gratuitous.} Proposition~\ref{prop:naive}
shows that ignoring informative left censoring inflates the total-time-on-test
estimator by the factor $(1+\beta)[\psi(2+\beta)-\psi(1)]$ --- already a
threefold error at $\beta=1$, and independent of the depth of right censoring.
The simulations reproduce the predicted bias at $n=1000$ to a relative accuracy
of $0.3\%$--$1.7\%$, the deviation decreasing with $n$. The freedom from
$\theta$ is directly visible: in both designs with $\beta=1$ the simulated limit
is $3.01$. One-sided methodology is therefore not merely suboptimal but
inconsistent in a two-sided informative setting.

\textbf{Robustness profile.} The conditions of the model fail in sharply
different ways, and our analysis quantifies each. The most useful finding is the
asymmetry inside \eqref{eq:model}: by Theorem~\ref{th:gfree} the right-censoring
half $1-G=(1-F)^{\theta}$ is never used by $\alpha_n$, so an arbitrary mechanism
of exit from observation inflates the RMSE by at most $21\%$ and produces no
bias at all (Section~\ref{ssec:gviol}), whereas a violation of the entry half
$K=N^{\beta}$ does bias the estimator (Section~\ref{ssec:kviol}). Since the
moment of entry into observation is usually governed by the study design while
loss to follow-up is not, the condition that matters is precisely the one the
investigator can secure. Misspecification of the \emph{lifetime law} is also
mild: the pseudo-true parameter departs very little from the mean lifetime
(to second order for the Weibull, linearly but with a small slope for the
gamma), the irreducible approximation distance stays below $0.04$ for $\pm10\%$
Weibull/gamma shape departures, and the parametric estimator remains
\emph{superior to the correctly specified semiparametric one} up to sample sizes
of $500$--$2000$ (Table~\ref{tab:robust}, Figure~\ref{Fig4}) --- which covers
the sample sizes at which a parametric analysis is realistically contemplated.
Misspecification of the entry mechanism is more serious (coarse violations at a
censoring depth of $50\%$ give asymptotic biases of $21$--$26\%$), but it
responds linearly and mildly to contamination ($\approx-0.21\varepsilon$) and
can be checked before use through the characterization of
Theorem~\ref{th:char}, as illustrated on the 6-MP data. The limits of that
protection are equally clear: by Table~\ref{tab:power} the power of the
diagnostic depends appreciably on the size of the violation --- at $n=300$ a
coarse violation ($\varepsilon=1$) is detected almost surely (power $1.00$)
while a mild one ($\varepsilon=0.25$) is largely missed (power $0.16$), and at
$n\lesssim100$ the power is low against violations of any size. For small and
moderate samples the entry condition must therefore rest on the study design
rather than on the diagnostic. We regard this combination --- graceful
degradation together with a built-in diagnostic --- as the principal practical
safeguard of the model.

\textbf{Limitations and extensions.} Three directions remain open. First, the
\emph{rate} of the normal approximation under heavy left censoring:
Theorem~\ref{th:clt} gives the limit law for all $\beta>0$, but Figure~\ref{Fig3}
shows that convergence slows appreciably as $\lambda$ decreases (skewness
$+3.84$ at $n=100$ for $\lambda=0.2$). A Berry--Esseen-type bound or an
Edgeworth correction would allow normal-theory confidence intervals to be
adjusted in small samples --- such a bound is already available in the one-sided
case \citep{CL87} --- and the identity of Lemma~\ref{lem:phi} appears a natural
starting point for such an analysis as well. Second, although the scope of the
construction \eqref{eq:alphan} is settled in Section~\ref{ssec:general},
families outside it remain open: the gamma and the lognormal --- even with known
shape --- do not belong to the proportional hazards class, since their
cumulative hazard cannot be written as $B(\vartheta)A(x)$, so for them numerical
maximization is unavoidable; for the Weibull with unknown shape a profile step
remains, with the inner step explicit. Third, the covariate and competing-risks
versions of the model \citep{Ab98} call for a similar parametric treatment.

\section{Conclusion}\label{sec:concl}

This paper has developed and studied a closed-form pseudo-MLE of the exponential
scale parameter in a model of informative random censorship acting from both
sides. The estimator is exactly Fisher-consistent, and remains so whatever the
right-censoring law --- so that only the entry mechanism need be informative ---
strongly consistent at every censoring depth, and asymptotically normal at all
those depths with explicit constants in polygamma form. The naive estimator that
ignores left censoring was shown to carry an explicit asymptotic bias,
independent of $\theta$, which renders it useless. The scope of the method was
determined analytically: an explicit solution exists precisely on the
proportional hazards class (Propositions~\ref{prop:phclass}
and~\ref{prop:converse}), and within that class every result reported here is
preserved verbatim at each $n$. An extensive Monte Carlo programme --- five
censoring designs, sample sizes from $30$ to $5000$, with diagnostics for bias,
RMSE, rate and normality --- confirmed the theory quantitatively. The estimator
operates within $98.6\%$--$99.4\%$ of the information bound of the model in the
designs considered and no worse than $96.2\%$ over a wider sweep; it reproduced
the asymptotic standard deviation to within $0.5\%$--$4.0\%$ at $n=1000$ and
better than $1.3\%$ at $n=5000$; and it showed a uniform advantage over the
semiparametric power-type estimator in sup-norm, reaching a factor of $3.7$
under heavy right censoring. The robustness analysis established that the
estimator degrades very slowly under Weibull and gamma departures of the
lifetime law --- remaining superior to the correctly specified semiparametric
alternative up to $n\approx500$--$2000$ even with a $10\%$ shape error, that is
at the sizes for which a parametric analysis is realistically contemplated ---
that it responds only linearly and mildly to violations of the informativeness
condition, and that it comes equipped with a diagnostic that allows
\eqref{eq:model} to be tested on the data through Theorem~\ref{th:char}, albeit
one that examines both halves of the model together and has low power in small
samples. The applications to the 6-MP leukaemia trial and to Nelson's
insulating-fluid experiment illustrate, respectively, the exact reduction of the
procedure to the classical censored-exponential analysis in the absence of left
censoring, and the recovery of a known ground truth under severe artificial
two-sided censorship. The third application --- the AIDS data on intravenous
drug users, the one example in which left censoring is \emph{genuine} ---
displays the practical value of the theory directly: the diagnostic rejects the
right-censoring half, which by Theorem~\ref{th:gfree} does not affect
$\alpha_n$, while the exponential shape of the lifetime law proves untenable and
calls for the proportional hazards extension of Section~\ref{ssec:general}.
Taken together, these results single out the two-sided informative censorship
model with exponential lifetime as possessing a rare combination of
solvability, efficiency and testable conditions, and its parametric estimator as
a construction that unites those three properties in a way one-sided and
non-informative alternatives do not.

\bmhead{Acknowledgements}
	The authors are deeply grateful to their teacher A.~A.~Abdushukurov, who
	introduced them to the two-sided informative model of random censorship and
	whose guidance and discussions shaped this work.

\section*{Statements and Declarations}

\paragraph{Competing interests.}
The authors have no competing interests to declare that are relevant to the
content of this article.

\paragraph{Funding.}
The authors did not receive support from any organization for the submitted
work.

\paragraph{Data availability.}
The three data sets analysed in this study are already in the public domain: the
6-MP leukaemia remission data are reproduced in full in
Section~\ref{ssec:sixmp} and originate from \citet{Fr63} and \citet{Ge65}; the
insulating-fluid breakdown times are reproduced in full in
Section~\ref{ssec:fluid} and originate from \citet{Ne72}; and the AIDS data on
intravenous drug users are distributed with the {\tt dblcens} package
\citep{dblcens} and were analysed by \citet{JG11}. No new data were generated.

\paragraph{Code availability.}
The code that generates every table and figure of this paper --- the Monte Carlo
programs, the numerical evaluation of the information bound and of the
pseudo-true parameter, and the scripts that reproduce the three applications ---
is available from the corresponding author upon reasonable request. All random
seeds are fixed, so that the reported figures are reproducible exactly.

\paragraph{Author contributions.}
All authors contributed to the conception and design of the study. D.~R.~Mansurov developed the theory and wrote the first draft; S.~B.~Bozorov carried out the Monte Carlo study and the numerical verification; A.~B.~Oltiyev performed the
analyses of the real data sets. All authors read, revised and approved the final manuscript.

\paragraph{Ethics approval.}
Not applicable: the study involves no new research on humans or animals, and
uses only previously published, anonymized data.


\makeatletter\if@filesw\immediate\write\@auxout{\string\citation{snay}}\fi\makeatother
\bibliography{references}

\end{document}